\documentclass[11pt]{amsart}
\usepackage{amssymb, amsmath, amsthm,ulem}
\usepackage[latin1]{inputenc}
\usepackage{graphicx}
\usepackage[hidelinks]{hyperref}
\usepackage{color}
\usepackage{ulem}
\usepackage{epstopdf}
\usepackage{cancel}
\usepackage{tikz}

\usepackage[a4paper, centering]{geometry}
\usepackage{color}
\usepackage{graphicx}
\usepackage{scalerel,stackengine}

\usepackage{mathtools}
\newtheorem{theorem}{Theorem}
\newtheorem{proposition}[theorem]{Proposition}

\newtheorem{lemma}[theorem]{Lemma}
\newtheorem{corollary}[theorem]{Corollary}

\theoremstyle{definition}
\newtheorem{definition}[theorem]{Definition}
\newtheorem{remark}[theorem]{Remark}

\definecolor{verde}{RGB}{20,150,100}
\definecolor{purple}{RGB}{200,30,200}

\newcommand{\ka}{\mathcal K}

\newcommand{\caone}{\ka_\infty}
\newcommand{\cafour}{\gamma}
\newcommand{\cafive}{\ka_1}
\newcommand{\casix}{\ka_3}
\newcommand{\caseven}{\ka_2}
\newcommand{\cakroger}{\ka_{kroger}}

\stackMath
\newcommand\reallywidecheck[1]{%
\savestack{\tmpbox}{\stretchto{%
  \scaleto{%
    \scalerel*[\widthof{\ensuremath{#1}}]{\kern-.6pt\bigwedge\kern-.6pt}%
    {\rule[-\textheight/2]{1ex}{\textheight}}%WIDTH-LIMITED BIG WEDGE
  }{\textheight}% 
}{0.5ex}}%
\stackon[1pt]{#1}{\scalebox{-1}{\tmpbox}}%
}

\def\R{\mathbb{R}}

\def\N{\mathbb{N}}

\newcommand{\sm}{\setminus}
\newcommand{\vps}{\varepsilon}
\newcommand{\Om}{\Omega}
\newcommand{\om}{\omega}

\newcommand{\sq}{\subseteq}
\newcommand{\ra}{\rightarrow}

\def \e{\varepsilon}
\def\ds{\displaystyle}

\newcommand{\nif}{{n \to + \infty}}
\def\RR{\mathbb{R}}

\newcommand{\abs}[1]{{\left|#1\right|}}
\newcommand{\norma}[1]{{\left\Vert#1\right\Vert}}
\newcommand{\vphi}{\varphi}

\begin{document}

\title[]{ 
  A sharp quantitative nonlinear Poincar\'e inequality  on convex domains  
}

\author[]{Vincenzo Amato, Dorin Bucur, Ilaria Fragal\`a}

\thanks{}

\address[Vincenzo Amato]{
Dipartimento di Matematica e Applicazioni ``R. Caccioppoli'', 
Universit\`a degli studi di Napoli Federico II \\ Via Cintia, Complesso Universitario Monte S. Angelo\\ 
80126 Napoli (Italy)}
\email{vincenzo.amato@unina.it}

\address[Dorin Bucur]{
Universit\'e  Savoie Mont Blanc, Laboratoire de Math\'ematiques CNRS UMR 5127 \\
  Campus Scientifique \\
73376 Le-Bourget-Du-Lac (France)
}
\email{dorin.bucur@univ-savoie.fr}

\address[Ilaria Fragal\`a]{
Dipartimento di Matematica \\ Politecnico  di Milano \\
Piazza Leonardo da Vinci, 32 \\
20133 Milano (Italy)
}
\email{ilaria.fragala@polimi.it}

\keywords{Neumann eigenvalue, sharp quantitative form, Poincar\'e inequality. }
\subjclass[2010]{35P15, 49R05} 
\date{\today}

%35P15 Estimates of eigenvalues in the context of PDEs
%49R05 Variational methods for eigenvalues of operators

\maketitle
\begin{abstract}  
 For any $p \in ( 1, + 
  \infty)$,  we  give a new inequality  for the  first nontrivial Neumann eigenvalue   $\mu _ p (\Om, \phi)$   of 
  the $p$-Laplacian on a convex domain  $\Om \subset \R ^N$ with a power-concave weight $\phi$. 
  Our result improves the classical estimate in terms of the diameter, first stated in a seminal paper by Payne and Weinberger: we add in the lower bound an extra term depending on the second largest John semi-axis of $\Om$ (equivalent to
a power of the width in the special case $N = 2$). The  power exponent in the extra term is sharp, and 
  the constant in front of it is explicitly tracked, thus
 enlightening  the interplay between  space   dimension, nonlinearity and  power-concavity. 
Moreover, we attack the stability question:  we prove  that, if $\mu _ p (\Om, \phi)$ is close to the lower bound,   then   $\Om$  is close to a   thin  cylinder, and  $\phi$   is close  to a function which is constant   along   its axis.    As intermediate results,  
  we establish a sharp $L^ \infty$ estimate for   the associated eigenfunctions,  and 
  we determine the   asymptotic   behaviour of    $\mu _ p (\Om, \phi)$   for varying   weights and     domains, including the case of collapsing geometries.  
\end{abstract}

\section{Introduction }
Given an open bounded convex set  $\Om \subset \R^N$    and  a  positive function  $\phi\in L ^ 1(\Om) $,  let 
\begin{equation}
	\label{mu2}
	\mu_p(\Omega,\phi):=\inf_{\substack{u\in W^{1, p}(\Om,\phi)\sm \{0\}\\ \int_\Omega \phi  |u | ^ { p-2} u =0}} \frac{\displaystyle{\int_\Omega \phi \abs{\nabla u}^p}}{\displaystyle{\int_\Omega \phi |u |^p}}\,,
	\end{equation}
 where  $p$ is a real exponent in $( 1, + \infty)$ and $W^{1, p}(\Om,\phi) $  is usual weighted Sobolev space.  
	When $ \phi \equiv 1$, $\mu_p(\Omega,\phi) $ (that we shall denote by $\mu_p(\Om)$)  is the first nonzero eigenvalue of the Neumann $p$-Laplacian on $\Om$, or the inverse of the optimal constant for the   Poincar\' e inequality in  $W ^ { 1, p}(\Om)$. 

In 1960, Payne and Weinberger  gave  a sharp lower bound for $\mu_2(\Om)$ in terms of the diameter  $D _\Om$   of the set $\Omega$ (see   \cite{PW, Bebendor}). 
Their inequality was extended to the case of the $p$-Laplacian eigenvalue $\mu _ p ( \Om)$ with $p \geq 2$ 
(see \cite{ENT}), and then  to the case of the weighted $p$-Laplacian $\mu _ p ( \Om, \phi)$ 
for any $p \in (1, + \infty)$ and $\phi$ log-concave (see \cite{FNT12}),   in which it reads  
\begin{equation}
	\label{paynewein}
	\mu_p(\Omega, \phi)\ge \Big  (\frac{\pi_p}{ D_\Om} \Big ) ^ p  \, , \qquad \pi _p:= 2 \pi \frac{(p-1) ^ {\frac 1 p}}{p \sin  \big (\frac \pi p \big ) } 
	,
\end{equation}

  We remind that $\pi _p$ agrees     with  the first positive zero of a generalized $p$-sine function (see \cite{Lind95}), and 
gives the optimal constant in the one-dimensional Poincar\'e inequality for $\phi \equiv 1$ (see \cite{BBCDG, RW99}), i.e., denoting by  $I _d$  an interval of length $d$, 
$$\Big (  \frac{\pi _p}{d} \Big ) ^p =\inf_{\substack{u\in W^{1, p}( I_d)\sm \{0\}\\ \int_{I_d}  |u | ^ { p-2} u =0}} \frac{\displaystyle{\int_{I_d}  \abs{u'}^p}}{\displaystyle{\int_{I_d}  |u |^p}}\,.$$

For $p = 2$,  and assuming the weight $\phi$ to be $(\frac 1 m )$-concave, in our previous paper \cite{ABF24} we improved the inequality \eqref{paynewein} into the following quantitative form, which in the particular case $N = 2$ and $\phi \equiv 1$,  had been conjectured in 2007 by Hang and Wang \cite{hangwang}:
\begin{equation}
		\label{quantitative3d}
		\mu_2(\Omega, \phi)\ge \Big  (\frac{\pi}{ D_\Om} \Big ) ^ 2  +\overline c \frac{a_2^2}{ D_\Om ^4}\,,
	\end{equation}
Here $a_2$ denotes the second largest John semi-axis of $\Om$   (see e.g. \cite{guzman} for the  classical definition of John ellipsoid),   and $\overline c$ is a 
positive constant depending only on $N +m$.    Later, a different proof has been proposed 
 in dimension $N = 2$  and for $\phi \equiv 1$ in a paper  posted on ArXiv  \cite{WangXu}.   For further bibliography about about Poincar\'e-type inequalities,  we refer the reader 
to the treatise \cite{Maz11} 
and the survey paper \cite{KN15}, and enclosed references (see in particular \cite{Kla, Valtorta}).

  In this paper we attack the problems of establishing an  inequality  of the type \eqref{quantitative3d}  in the nonlinear setting, and of studying its stability, both with respect to the domain and the weight.  The main results are stated in Theorems \ref{t:Ngap} and \ref{t:stability} below. 

 In Theorem \ref{t:Ngap} we provide the 
sharp estimate with extra term related to  flatness. We stress that  it   
 requires a conceptually different approach with respect to the linear case
 and that its purpose goes beyond, as it contains an explicit  value for the constant,
in terms of the dimension, nonlinearity and power concavity.  We prove:

 \begin{theorem}\label{t:Ngap}
Let  $N\ge 2$ and $m \in \N \setminus \{ 0 \}$. There exists a  positive  constant  $ \ka_0 = \ka_0 (p,  N , m) $ such that, 
 for every open bounded convex domain $\Omega$ in $\R ^N$ with diameter $D_\Om$
and John ellipsoid of semi-axes  $a_1 \geq \dots \geq a_N$,  and 
every  $(\frac {1}{m})$-concave function $\phi : \Om \to (0,+\infty)$ 
we have
	\begin{equation}
		\label{quantitative3dp}
		\mu_p (\Omega, \phi)\ge \Big  (\frac{\pi_p}{ D_\Om} \Big ) ^ p   +     \ka_0 \frac{a_2^2}{ D _\Om ^{p+2} }.
		\end{equation}
For $N = 2$, and general $p$, $m$, an  explicit  value of the constant $ \ka_0 $ is given by 
\begin{equation} \label{f:costantefinale}   \ka _0:=  
 \frac{ \cafive  (\caseven) ^ 2 }{6(7 \cdot 16 \cdot 256  )^2}, 
\end{equation}    
where
$\ka _1$ and $\ka _2$ are  given respectively   in \eqref{f:cafive}  and  in \eqref{f:caseven} .
\end{theorem}

\medskip

The proof  of Theorem \ref{t:Ngap} is given in Section \ref{sec:proof1}.   The reason why the nonlinear setting requires a completely different  approach with respect to the linear one is that,  
 for any $p\not=2$, in one dimension of the space,  there is no equivalence between the weighted Neumann problem and a Dirichlet problem associated with a suitable measure potential. Such equivalence played a key role both in original proof by Payne-Weinberger proof and in the quantitative analysis we carried over in \cite{ABF24}. 
Therefore   we have  to deal directly with the Neumann problem.   To that aim,  
 our main novel ingredient   
  is the one-dimensional result stated in Proposition \ref{p:1d},   which represents the  hard core of the proof, and  
is somehow inspired from the approach adopted in \cite{ENT, FNT12}. 
While the passage to higher dimensions is similar to the linear case (so when  it is the case we shall refer the reader to \cite{ABF24}), 
the proof  also   requires   further new  intermediate results,
	some of them of possible independent interest.    Specifically, 
	 we need to study  the 
	 asymptotical behaviour of 
Neumann eigenvalues under variation of the domain and of the weight (see Section \ref{section2}), 
and to prove a sharp $L^ \infty$ estimate 
for eigenfunctions of the weighted Neumann $p$-Laplacian (see Section \ref{sec:infty}).  

Below we give several observations about  the  inequality  \eqref{quantitative3dp}.

\begin{remark}[sharpness of the exponent  in the extra term]    The power $2$ of the second John semiaxis in   the inequality \eqref{quantitative3dp}   cannot be improved, independently of  $p$.   This can be seen through the following example,  given for simplicity in dimension $N = 2$ and for $\phi = 1$. Let $\Om _\e = (0, 1)\times (0, \e)$. Taking $u (x, y) = \overline u (x)$ as a test function for $\mu _p (\Om _\e)$,  where $\overline u$ is the  first eigenfunction for $\mu _p ( I)$,  
we obtain 
\begin{equation}\label{f:su}
\mu_p (\Om _ \e) \leq  (\pi _ p ) ^ p\,.
\end{equation}
Assume by contradiction that, for some $\delta, \overline c  >0$, 
\begin{equation}\label{f:giu}\mu_p (\Om _ \e) \geq  \Big  (\frac{\pi_p}{ D_{\Om_\e}} \Big ) ^ p   +    \overline c \frac{a_2^{2 -\delta}}{ D _{\Om  _\e}^{p+2 - \delta} }
 \,.\end{equation}
  By combining   the inequalities \eqref{f:su} and \eqref{f:giu},   and  taking into account that $D _{\Om _\e} = ( 1 + \e ^ 2) ^{ \frac{1}{2} } $ and $a _ 2  \approx \e$,  we obtain $\e ^ 2 \geq C \e ^ { 2 - \delta}$ for some positive constant $C$,  which is a contradiction by the arbitrariness of $\e >0$. 
 \end{remark}

\begin{remark}[power-concavity versus log-concavity]    If   the weight $\phi$ is merely log-concave,  
the inequality \eqref{quantitative3dp} is in general false.
This can be seen through a counterexample, given for simplicity  in dimension $N = 2$ and for  $p=2$. We consider the sequence of log-concave weights defined  by 
$$
\phi_n  (x, y)= n e^{-(ny)^2}\qquad \forall (x, y) \in \R ^2
$$
and we claim that, if $B_1$ is the Euclidean unit ball in $\R^2$,
$$\lim _{n \to +\infty} \mu _ 2 (B_1, \phi _n)  = \frac{\pi^2}{4}= \frac{\pi^2}{D_{B_1 } ^ 2}\,.$$ 
By symmetry, we can  
 take $u(x,y)= \sin (\frac{ \pi}{2}x)$ as test function for $\mu_2({ B_1},\phi_n)$.
Then, 
 passing to the limit as $n \to + \infty$,  by dominated convergence we obtain
 $$
\frac{\pi^2}{4}\le \limsup_{n \to + \infty}  \mu_2({B_1},\phi_n) \leq\frac{\pi^2}{4} \frac{\displaystyle{\int_{B_1 }  \cos^2 (\frac{ \pi}{2}x)   \phi_n   }}{\displaystyle{\int_{B_1 }  \sin^2 (\frac{ \pi}{2}x)   \phi_n   }}
= \frac{\pi^2}{4} \frac{\displaystyle{\sqrt{\pi} \int_{-1}^1 \cos^2 (\frac{ \pi}{2}x)\, dx}}{\displaystyle{\sqrt{\pi} \int_{-1}^1 \sin^2 (\frac{ \pi}{2}x)\, dx}} =\frac{\pi^2}{4}\,.
 $$
\end{remark} 

\begin{remark}[relaxed power concavity assumption]
Since the class of $\alpha$-concave functions is monotone decreasing under inclusion with respect to $\alpha$, 
Theorem   \ref{t:Ngap}       continues to hold 
  for any $\alpha$-concave positive function $\phi$, being $\alpha\in (0, 1]$, and in this case the constant $\overline c$ depends only on $p$ and $N + \big [  \frac{1}{\alpha}  \big  ]+ 1 $, where $[\, \cdot \,  ]$ denotes the integer part. 
\end{remark}

  \begin{remark}[on the constant in the extra term]  In tracking the expression of  the constant $\mathcal K _0$ in the inequality \eqref{quantitative3dp}, we do not aim at optimality. However, we believe that the expression
 given in \eqref{f:costantefinale} may be of interest for later applications,  in particular for numerical
purposes such as controlling the error in polynomial approximations in Sobolev spaces (see \cite{Ver99}).  
Let us point out that, according to  a Blaschke diagram obtained via numerical experiments  for $N=2, p=2, \phi\equiv 1$, 
the optimal value of $\ka _0$   should lie between $5$ and $6$  (which is quite far from the explicit value coming from \eqref{f:costantefinale}). 
Determining such optimal value, as well as  understanding the geometries leading to it,   remains a  very challenging question, which falls  out of the scopes of this  paper.  
  \end{remark}

\bigskip
 
We now turn attention to the  stability question, which can be formulated as follows: is it true that, 
if  $\mu_p(\Om, \phi)$ is close to $\big ( \frac{\pi_p}{D_\Om} \big ) ^ p$ for a convex set $\Om$,
in some sense $\Om$ is geometrically close to a line segment and $\phi$ is close to a constant function?
 This kind of question is not new, e.g. something similar was asked for Polya's functional related to the Torsion rigidity in \cite{AMPS}. 
Below we provide a partial affirmative answer for $\mu_p(\Om, \phi)$, which is 
not quantitative, but just asymptotical: it states roughly  that,  whenever $\mu _ p (\Om, \phi)$ approaches $( \frac  {\pi_p }{D _\Om} ) ^ p$, the convex set  
$\Om$ tends to become a thin cylinder  and $\phi$ approaches  locally uniformly
  a function which is constant along the axis of such cylinder. 

\begin{theorem}\label{t:stability} 
	Let $N\ge 2$ and $\alpha \in (0, 1]$.   Let $\{ \Omega_n \}$ be a sequence of open bounded convex domains in  $\R^N$ having
diameter $1$,
and  John ellipsoid
centred at the origin, with semiaxes along  the cartesian directions,   of lengths
	$ a_1^n \ge \dots\ge a_N^n$.  Let $\{\phi _n\}$ be a sequence  of   positive    $\alpha$-concave  weights. Assume that
	$$
	 \mu_{p}(\Omega_n, \phi _n)   \to    \pi_p^p\,. 
	$$
Then, up to subsequences,  we have 

\begin{itemize}
	\item [(i)]  $ a_i^n  \to 0 \quad \forall   i \in \{2, \dots, N\}\,,$
	\smallskip 
	
         \item [(ii)] There exist affine transformations $T_n: \R^N\to \R^N$ such that $\{T_n(\Om_n)\}$ converges in the 
		 Hausdorff metric to a cylinder of the form $I \times K$,   being   $I= \big ( - \frac 1 2 , \frac 1 2 \big )$  and $K$ an  
	open bounded convex domain  in $\R^{N-1}$. 
	
	\smallskip 
	\item [(iii)]  $  \frac{\phi_n\circ T_n^{-1} }{\|\phi_n\|_\infty}\to  \phi$ locally uniformly in $I \times K$, where $\phi$ is constant in the first variable.
\end{itemize}
	\end{theorem}

\medskip 
 The proof of Theorem \ref{t:stability} is given in Section \ref{sec:proof2}.   Basically, it relies on the asymptotical results given in Section \ref{section2},  but it demands also a specific geometric argument  which involves the equality case of Brunn-Minkowski inequality.

\medskip

\section{Convergence of Neumann eigenvalues  for varying domains and weights }
\label{section2}

The stability of the Neumann spectrum under domain variation has been widely studied in the literature (see for instance \cite{Ar95, BuLa07,Ftouhi}). However, for our specific purposes, which include dealing with varying weights (possibly non uniformly elliptic) and singular geometric behaviors  (possibly degenerating convex sets), we did not find any suitable reference in the literature. Thus, 
  in the next result we establish the asymptotics of  
Neumann eigenvalues on varying domains and 
with varying  weights, including  the singular case of 
collapsing domains
(namely, domains which asymptotically ``loose'' some dimensions). 
  In particular, we are going to enlighten the interplay between the concavity \ exponent   of the weights and 
  the number of  collapsed dimensions.

Let us mention that some general upper semicontinuity results for density varying Neumann eigenvalue problems can be found in  \cite{BMO23}; nevertheless,  they do not apply in our framework since the definition of   degenerate eigenvalues adopted therein is slightly different.

As a   preliminary    step, let us  rephrase for any $p \in ( 1, + \infty)$
 the following lemma, which was obtained for $p = 2$ in \cite[Lemma 42]{ABF24}.

\begin{lemma} 
For any open bounded convex domain $\Om \subset \R ^N$, any   positive weight 
$\phi \in L ^ 1 (\Om)$
which is $\alpha$-concave  for some $\alpha \in (0,1]$, and any $p \in ( 1, + \infty)$, the
embedding 
$W^{1,p} (\Om, \phi) \hookrightarrow L^p(\Om,\phi)$
is compact. 
\end{lemma} 
\proof  
Let $m \in \N \setminus \{ 0 \}$ such that $\frac 1m \le \alpha$.  Let $\widetilde{\Omega} \sq \R^{N+m}$ be defined by
	 $$
	 \widetilde{\Omega } = \left\{ (x,y) \in \R^{N}\times \R^{m} : \, x \in \Omega ,\,  \norma{y}_{\R^m} <\omega_m^{-1/m} \phi^{1/m}(x)\right\}\,,
	 $$
 where $\om_m$ stands for the volume of the unit ball of $\R^m$. 
 
	 Since $\phi^{ \frac{1}{m}}$ is  also concave  by the choice of $m$,  $\widetilde{\Omega}$ is  open  and convex, so that 
	 	$W^{1,p} (\widetilde{\Omega})$ is compactly embedded into $L^p(\widetilde{\Omega})$. 
Now, if  $\{u_n\}$ is a bounded sequence in $W^{1,p}(\Om,\phi)$, setting  
	$
	\tilde{u}_n(x,y):= u_n(x)$
	 we have
	 \begin{align*}
	\int_{\widetilde{\Omega}} | \tilde{u}_n| ^p(x,y) \, dxdy &= \int_{\Omega} |u_n|^p(x)  \phi (x) \, dx ,\\
		\int_{\widetilde{\Omega}} \abs{\nabla_{x,y}\tilde{u}_n}^p(x,y) \, dxdy &= \int_{\Omega} \abs{\nabla_x u_n}^p(x)  \phi(x) \, dx.
	 \end{align*}
	Then, up to subsequences, $\tilde u _n$ converges weakly in 
	$W^{1,p} (\widetilde{\Omega}) $ and strongly in $L ^ p ( \widetilde \Om)$ to a function $\tilde u \in W^{1,p}(\widetilde{\Omega})$.  
	 Since $\tilde u$ is constant in the  $y$ variable, the function $u(x): = \tilde u (x, y)$ belongs to $W^{1,p}(\Omega,\phi)$, and $u _n$ converges strongly to $u$ in $L ^ p (\Om, \phi)$. 
\qed

\bigskip

   We 
are now ready to give  the convergence result. We state it for  $p=2$,   
since in this case we can describe the behavior of the full spectrum of the Neumann Laplacian. 
 For arbitrary $p$,  the same arguments will  give the asymptotic behaviour of the first nonzero 
Neumann $p$-Laplacian eigenvalue.
 
Limited to the remaining of this section, in order to avoid cumbersome notation   with double indices,
 we simply denote by $\mu _k(\Om, \phi)$, $k\in \N$,  the 
sequence of Neumann Laplacian eigenvalue  with weight $\phi$ (tacitly meaning that we are taking $p = 2$).

 \begin{proposition}[Asymptotics of Neumann Laplace spectrum]
	\label{abf209}
	Let $\{\Omega_n\}$ be a sequence of open bounded convex sets in $\R^N$,  having their John ellipsoid centred at the origin, with semiaxes along  the cartesian directions, of lengths
	$ a_1^n \ge \dots\ge a_N^n$.   Assume that, for some  $l \leq N$,   it holds 
	$$
	\lim_\nif (a^n_1, \dots, a^n_l, a^n_{l+1} , \dots a^n_N)=(a_1, \dots, a_l, 0 , \dots 0)\,,  \qquad \text{ with }\  a _l >0\,.
	$$
	Moreover, let  $\phi_n:\Om_n \to (0, +\infty) $ be a sequence of functions, with $\|\phi_n\|_{L^\infty} =1$, which are $(\frac 1 m)$-concave for some $m \in \N \setminus \{ 0 \}$. Then
	
	\begin{itemize}
\item[(i)] There exist an open convex set $\omega \subset \mathbb{R}^l$ and a $\big (\frac{1}{N-l+m}\big )$-concave function $  \Phi : \omega \to \mathbb{R}_+$ such that 
	\begin{equation}\label{abf2061}
		\mu_k(\Omega_n,\phi_n) \to \mu_k(\omega,   \Phi  ) \qquad  \forall k \in \N \,. 
	\end{equation}
	\item[(ii)]
For every $k \in \N$,   if $u_{n,k}$ are   $L^2(\Om _n,  \phi_n )$-normalized, pairwise orthogonal, eigenfunctions for $\mu_k(\Om_n)$,  for a.e.\ $(x_1, \dots, x_l) \in \om$ and a.e.\  $(x_{l+1}, \dots, x_N)\in \R^{N-l}\! \! $  such that $(x_1, \dots, x_l, x_{l+1}, \dots, x_N) \in \Om_n$, we have
\begin{eqnarray}
	& \lim\limits_{ n \to + \infty} (a^n_{l+1}\dots a^n_N)^\frac 12 \,  u_{n,k}(x_1, \dots, x_l, x_{l+1}, \dots, x_N) =  u_{k} (x_1, \dots, x_l), &\label{abf202}
	\\  \noalign{\medskip} 
	&  \|u_k\|_{L^\infty(\om)} \le \liminf \limits_\nif   (a^n_{l+1}\dots a^n_N)^\frac 12  \|u_{n,k}\|_{L^\infty(\Om_n)}. 
	& \label{abf208}
\end{eqnarray}
where $u_k$ is a $L^2(\om,  \Phi  )$-normalized eigenfunction for $\mu_k(\om,  \Phi )$.
\end{itemize} 
\end{proposition}

\begin{proof} 
	For every $n \in \N$, let us define the map     $T_n : \R^N\ra \R^N$  by
	$$
	T_n	(x_1, \dots, x_{ l  }, x_{  {l+1}   } , \dots x_N):= (x_1, \dots, x_ {  l  },  \frac{1}{a^n_{l+1}}  x_{l+1} , \dots,  \frac{1}{ a^n_{ N  }} x_N)
	$$
	There exist   nonempty balls $B _r, B _ R \subset \R ^N$   such that,  for every $n\in \N$, we have
	\begin{equation}\label{abf205}
		B_r \subset   T _n  (\Omega_n)\subset B_R.
	\end{equation}
	Let us consider the following  convex sets of $\R^{N+m}$:
\begin{equation}\label{f:tildeomegan}
	\widetilde{\Omega}_n  := \big \{ (x,y) \in \R^{N+m} \ :\  \, x \in    T_n(\Omega_n) , \, \norma{y}_{\R^m} <\omega_m^{-\frac 1 m} \phi_n^{\frac{1}{m} }(T_n^{  -1}(x))  \big \}.
\end{equation}
We have 	\begin{equation}\label{abf206}
	  B_r\times  B_{r'}  \subset   \widetilde \Omega_n  \subset  B_R  \times  B_{R'},
\end{equation}
for some nonempty balls $ B_{r'}, B_{R'} \subset \R^m$ .

	Up to extracting a subsequence, one can assume that  the sequence   $\{1_{ \widetilde \Omega_n}\} $ converges to $1_{\widetilde \Omega}$  in $L^1(\R^{N+m})$, where $\widetilde \Om$ is a nonempty open convex set in $\R ^{N+m}$. 
	Moreover, we have that  $\{ \widetilde \Om_n\}$ converges to $\widetilde \Om$ both in the Hausdorff and Hausdorff complementary topologies (see \cite{bubu} for the terminology), and it holds $B_r \times B_{r'}  \subset   \widetilde \Omega   \subset B_R  \times   B_{R'}$.       Notice also that $\widetilde \Om$ can be written as
	$$
	\widetilde{\Omega}  = \big \{ (x,y) \in \R^{N+m} \ :\  \, x \in  \Omega, \, \norma{y}_{\R^m} <\omega_m^{-\frac 1 m} \phi^{\frac{1}{m} }( x)  \big \}\,,
$$ 
where $\Om$ is the limit of $T _n (\Om_n)$ is Hausdorff distance and  $\phi$ is the pointwise limit of $\phi _n \circ T_n ^ {-1}$ in $\Om$.

	Let us set $\omega:= \pi_{\mathbb{R}^l \times \{0\}^{N+m-l}} (\widetilde{\Omega}) $, namely $\om$ is the open convex subset of $\R ^ l$ given by 
	$$
	\omega =  \big \{ (x_1,\dots, x_l)\, : \exists \;   (x_{l+1} , \dots x_{N+m} ) \in \R ^ {N- l+m}   \mbox{ s.t. } 	(x_1, \dots, x_l, x_{l+1} , \dots x_{N+m}) \in \widetilde{\Omega}  \big \}.
	$$
	Moreover, we define the function $  \Phi:    \om \to \RR_+$ by 
 	$$
	  \Phi  (x_1,\dots, x_l)= \mathcal{H}^{N-l+m} \big ( \big \{ y\in  \widetilde{\Omega}\, :\,  \; y_i=x_i \  \forall i=1, \dots, l\big \} \big ).
$$
By the convexity of $\widetilde \Om$  and Brunn-Minkowski Theorem,   the function $  \Phi $ is  $\big (\frac{1}{N-l+m}\big )$-concave.
	
	\smallskip
	In order to get the upper semicontinuity of eigenvalues, let $\{u_0,\dots,u_k\}$ be $L^2(\om,   \Phi  )$-normalized, pairwise orthogonal,  eigenfunctions associated with $\mu_0(\omega,   \Phi  ), \dots, \mu_k(\omega,   \Phi  )$.
	For each $j = 0,\dots, k$, let $\tilde u_j : \widetilde \Om \to \R$ be  defined by
	$$\tilde u_j(x_1, \dots, x_l, x_{l+1}, \dots, x_{N+m}):= u_j(x_1, \dots, x_l).$$
	Then $\tilde u_j \in H^1(\widetilde \Om)$ and, by separation of variables, $\int_{\tilde \Om}\tilde u_j ^2  =1$,   $\int_{\tilde \Om}|\nabla \tilde u_j |^2  = \mu_j(\omega,   \Phi  )$. 
	
	From the Hausdorff convergence   of $\widetilde \Om _n$ to $\widetilde \Om$,   there exists an infinitesimal sequence $\vps_n>0$ such that $(1+\vps_n) \widetilde \Om \supseteq \widetilde \Om_n$. Defining $\tilde  v_{n,j}(x)= \tilde  u_j(\frac{1}{1+\vps_n} x)|_{\widetilde \Om_n}$, we get that $\tilde v_{n,j} \in H^1(\widetilde \Om_n)$ and that $(1_{\widetilde \Om_n}\tilde v_{n,j}, 1_{\widetilde \Om_n}  \nabla \tilde v_{n,j} )$ converges strongly in $L^2(\R^{N+m})$ to $(1_{\widetilde \Om } \tilde u_j, 1_{\widetilde \Om } \nabla \tilde u_j)$.

	Then we can use $\widetilde S_{k}^n:=\text{span}\{ \tilde v_{n,0}(x_1, \dots x_N), \dots, \tilde v _{n,k}(x_1, \dots x_N)\}$ as test  for $\mu_k(\Omega_n,\phi_n)$, which yields		
	$$
	\begin{aligned}
		\limsup_\nif \mu_k(\Omega_n,\phi_n)& \leq 	\limsup_\nif \max_{\substack{v \in  \widetilde  S_k^n \\ \norma{v}_{L^2(\Omega_n,\phi_n)}=1}}\int_{\Omega_n} \abs{\nabla v}^2\phi_n	 \\&=     \limsup_\nif \max_{\substack{\tilde v \in \text{span}\{ \tilde v_{n,0}, \dots \tilde v _{n,k}\}\\ \norma{\tilde v}_{L^2(\widetilde \Omega_n)}=1}}\int_{\tilde \Omega_n} \abs{\nabla \tilde v}^2 \\ &
		=     \max_{\substack{\tilde u \in \text{span}\{ \tilde u_{0}, \dots \tilde u _{k}\}\\ \norma{\tilde u}_{L^2(\widetilde \Omega )}=1}}\int_{\tilde \Omega} \abs{\nabla \tilde u}^2 = \max_{\substack{u \in  \text{span} \{u_0,\dots,u_k\}\\ \norma{u}_{L^2(\omega,   \Phi  )}=1}}	 {	\int_\omega \abs{\nabla u}^2 } \phi= \mu_k(\omega,   \Phi  ),
	\end{aligned}
	$$
	hence
	\begin{equation}
		\label{limsupmu}
		\limsup_ \nif \mu_k(\Omega_n,\phi_n)  \leq \mu_k(\omega,   \Phi  ).
	\end{equation}

	To get  the lower semicontinuity of  eigenvalues, let $S_k^n:=\text{span}\{u_{n,0},\dots,u_{n,k}\}$ be $L^2$-normalized, pairwise orthogonal, eigenfunctions associated with $\mu_0(\Omega_n,\phi_n), \dots, \mu_k(\Omega_n,\phi_n)$. 
	For each $j$,  let us introduce the function $\tilde v_{n,j}$ defined on $\widetilde{\Omega}_n$ by  
	$$
	\tilde v_{n,j}(x_1,\dots,x_N,x_{N+1},\dots x_{N+m})= (a^n_{l+1}\dots a^n_N ) ^ {\frac{1}{2}} \,u_{n,j}(T_n^{  -1 }(x_1,\dots,x_N)).
	$$
	
	By direct computation, we see that  $ \ds \int_{\tilde \Om_n} \tilde v_{n,j_1} \tilde v_{n,j_2} dx =\delta_{j_1j_2}$, 
	and
	$$
	\begin{aligned}
		& \int_{\tilde \Om_n} \Big | \frac{\partial \tilde v_{n,j}}{\partial x_i}\Big |^2 = 
		\int_{  \Om_n } \Big | \frac{\partial u_{n,j}}{\partial x_i} \Big | ^2 \phi_n   
 		\qquad &\forall  i=1, \dots, l ,
		\\ \noalign{\smallskip} 
		&  \int_{\tilde \Om_n} \Big | \frac{\partial \tilde v_{n,j}}{\partial x_i}\Big |^2 =  (a^n_i)^2	\int_{  \Om_n } \Big | \frac{\partial u_{n,j}}{\partial x_i}\Big |^2  \phi_n  
		\qquad &\forall i=l+1, \dots, N\,,
		\\ \noalign{\smallskip} 
		&  \int_{\tilde \Om_n} \Big | \frac{\partial \tilde v_{n,j}}{\partial x_i}\Big |^2 =  0
		\qquad &\forall i=N+1, \dots, N+m\,. 
	\end{aligned} 
	$$

	By  \eqref{abf205},   we can extend  $\tilde v_{n,j}$ to functions in $H^1(B_R \times  B_{R'}  )$ which are uniformly bounded, and hence up to subsequences converge to some limit $\tilde  {v}_j$, weakly in 
	$H ^ 1 ( B _R\times   B_{R'} )$ and strongly in  $L ^ 2 ( B_R   \times   B_{R'})$. The limit functions $\tilde v _j$ do not depend on $(x_{l+1} , \dots x_{N+m}) $, since, for   
	$i= l+1, \dots, N$, we have
	$$
	\int_{\widetilde \Om} \Big | \frac{\partial \tilde{v}_j}{\partial x_i}\Big | ^2  \leq \liminf_{n \to + \infty} 	\int_{\widetilde \Om_n} \Big | \frac{\partial \tilde v_{n,j}}{\partial x_i}\Big | ^2= \liminf_{n \to + \infty}\  (a_{i}^n)^2	\int_{  \Om_n } \Big | \frac{\partial u_{n,j}}{\partial x_i}\Big | ^2 \phi_n   
 = 0\,,
	$$
	and for   
	$i= N+1, \dots, N+m$, we have
	$$
	\int_{\widetilde \Om} \Big | \frac{\partial \tilde{v}_j}{\partial x_i}\Big | ^2 \leq \liminf_{n \to + \infty} 	\int_{\widetilde \Om_n} \Big | \frac{\partial \tilde v_{n,j}}{\partial x_i}\Big | ^2= 0\,.
	$$
	Then, if we define, for $(x_1, \dots, x_l, x_{l+1} , \dots x_{N+m}) \in \widetilde{\Omega}$, 
	$${v}_j(x_1, \dots,x_l) := \tilde{v}_j(x_1, \dots, x_l, x_{l+1} , \dots x_{N+m}) \,,$$ 
	and we set $S_k:= \text{span}\{ {v}_0, \dots ,  {v}_k\}$, 
	we have
	\begin{equation}\label{ambu31}
		\begin{aligned}
			\mu_k(\omega,\phi) &\leq  \max_{\substack{  v \in S_k\setminus\{0\} \\ \norma{  v}_{L^2(\omega ,   \Phi  )}=1}}\int_\omega  \abs{\nabla   v}^2   \Phi   =\max_{\substack{\tilde v  \in \text{span}\{\tilde {v}_{0}, \dots , \tilde{v}_{k}\} \\ \norma{\tilde v}_{L^2(\widetilde \Om )}=1}}\int_{\widetilde \Om} \abs{\nabla \tilde v}^2
			\\
			& \leq	\liminf_\nif  \max_{\substack{\tilde v  \in \text{span}\{\tilde {v}_{n,0}, \dots , \tilde{v}_{n,k}\} \\ \norma{\tilde v}_{L^2(\widetilde \Om_n )}=1}} \int_{\tilde \Om_n} \sum_{i=1}^l \Big | \frac{\partial \tilde v_{n,j}}{\partial x_i}\Big | ^2 +  \sum_{i=l+1}^N \frac{1}{(a_i^n)^2} \Big | \frac{\partial \tilde v_{n,j}}{\partial x_i}\Big | ^2   \\
			&
			=    \liminf _\nif    \max _{\substack{  u  \in  S_k^n \\ \norma{u}_{L^2( \Om_n,\phi_n )}=1}}  \int_{\Om_n} |\nabla u|^2 \phi_n =
			  \liminf_\nif    \mu_k(\Omega_n,\phi_n).
		\end{aligned}
	\end{equation}

	Finally,    from \eqref{limsupmu} and  \eqref{ambu31}, we see that 
	   \eqref{abf2061}   holds true, so that all the inequalities in  \eqref{ambu31} hold with equality sign, which implies that, for every $j = 0, \dots, k$,   $v_j $ has  to be a $j$-th eigenfunction for  $\mu _j(\om,   \Phi  )$ (and in particular  \eqref{abf202} and \eqref{abf208} hold true).
\end{proof}

 \bigskip 
 
 \begin{remark}
  \label{abf210} 
  Replacing the Laplace operator by the $p$-Laplace operator for $p \in ( 1, + \infty)$, with $p \neq 2$, 
   the statement of 
Proposition \ref{abf209} above remains true 
only for $k=1$, and with power $\frac 1p$ in place of $\frac 12$ in \eqref{abf202}-\eqref{abf208}.     The proofs are exactly the same, and 
the limitation to the case $k=1$ is of course required in order to have the definition of the first nontrivial  Neumann eigenvalue in terms of the Raleigh quotient.  
\end{remark}

\section{$L ^ \infty$ estimate for Neumann eigenfunctions}\label{sec:infty}

  In this section we prove the following result, which to the best of our knowledge is not covered by the literature about a priori
estimates for so lutions to $p$-Laplace type equations (see e.g. \cite{CM14}).  

 Below and in the remaining of the paper,  when a constant 
is defined as a function depending on  some of the parameters $p,N, m$, such dependence will be omitted, for notational simplicity,  during the subsequent computations. In case any ambiguity may arise, we shall specify the value of $p, N, m$ at which the constants are computed.

\begin{proposition}\label{ambu25}
There exists a  positive constant  $ \caone=  \caone (N)$ such that, 
for every open bounded convex domain $\Omega\subset \R ^N$ with diameter $D_\Om$,  
a first eigenfunction 
$ \overline u$  associated with the Neumann eigenvalue $\mu_p (\Omega)$ satisfies
	\begin{equation}\label{ambu40}
		\norma{\overline u}_{L^\infty(\Om)} \leq  \caone  \mu _p (\Om)^{\frac{N}{p}}  \frac{D _\Om ^N}{\abs{\Omega}^{\frac{1}{p}}}
		\norma{\overline u}_{L^p(\Om)}  \,. 		
	\end{equation}    
	The value of $\caone$ can be explicitly given as 
 	\begin{equation}\label{f:caone} \caone =  \caone (N):=    2 ^  {  2N }  \om _N ^ {-1 }     \big (  1+ 4 N ^ { N +1}  \big )   ^ { N  }  \Big( \frac{N}{N-1}\Big)^{N(N-1)} 
				\,.\end{equation}    
\end{proposition}

 \begin{remark} \label{abf.1.02}By inspection of the proof,  a refined form of the inequality \eqref{ambu40} holds if one allows  the constant to depend also on $p$,
i.e., $\caone$ could be be replaced by the smaller constant  
		\begin{equation}\label{f:k1tilde}  {{\widetilde \caone } (N,p):= }   
		2 ^  {  (\frac{p+1}{p} ) N } N ^N     \big ( N ^ {-1} \om _N ^ {-\frac{1}{N} }  (1+ 4 N ^ { N +1})  \big ) ^ { \frac{N}{p}  }  \beta (N, p)	
		\,,\end{equation} 
		with 
		$$
		 \beta (N, p)  := \prod_{k=0}^\infty\left(\frac{ (\frac{N}{N-1} ) ^{kp}}{p \big ( \frac{N}{N-1}  \big ) ^ k-p+1}  \right)^{\frac{1}{p }\big ( \frac{N-1}{N}  \big ) ^ k} \
		 		 $$ 	
		   However, since we do not seek for optimal constants in favour of simplicity, 	 		 in the sequel we shall keep using the inequality \eqref{ambu40} with $\caone$ given by \eqref{f:caone}.  
		  \end{remark}

 \begin{remark} By inspection of the proof,  Proposition \ref{ambu25} continues to hold  for higher order eigenfunctions  of the Neumann $p$-Laplacian,   meant as functions $u \in W^{1,p}(\Om)$ which satisfy 
the weak form of the equation $-\Delta_p u = \mu |u|^{p-2} u$ in $[W^{1,p}(\Om) ]'$.  

	\end{remark} 
 
As a consequence of Proposition \ref{ambu25}, we are going to prove  that  a $L ^ \infty$-estimate holds as well for {\it weighted} Neumann 
$p$-Laplacian eigenfunctions:

\begin{corollary}\label{c:collapse} 
For   any open bounded convex set $\Om \subset \R ^N$ and    
any positive  $\frac1m$-concave weight $\phi$    on $\Om$,      
a first eigenfunction 
$ \overline u$  associated with the Neumann eigenvalue $\mu_p ( \Om, \phi)$ satisfies 
 \begin{equation}\label{f:inftybound}
 	\norma{ \overline u}_{L ^ \infty( \Om)} \leq   \caone     \mu _p ( \Om , \phi)^{\frac{ N+m  }{p}} \frac{D_\Om ^ { N+m}}{{\left (\int_{\Om} 
\phi \right )}^{\frac 1 p}} \norma{\overline u}_{L^p(\Om ,\phi )}\,,
\end{equation}
 where $\caone$ denotes the constant defined by \eqref{f:caone}, computed at $N+m$. 
	\end{corollary}

 \smallskip 
For the proof of Proposition \ref{ambu25}, we need the following preliminary lemma.

\begin{lemma}[Estimate of the trace constant in $W ^ {1,1}(\Om)$] \label{l:trace} 
Let $B _1$ and $B _N$ are balls of radii $1$ and $N$ in $\R^N$, and let 
 $\Om$ be a convex set such that $B _ 1\subset \Om\subset B _N$. For every function $u \in W ^ { 1,1} (\Om)$, it holds
\begin{equation}\label{f:tconstant} 
\inf \Big \{   \frac{  \int _\Om |u| + \int _ \Om |\nabla  u |  }{\int _{\partial \Om}  |u| }  \ :\ u \in W ^ { 1, 1} (\Om) \setminus W ^ { 1, 1} _ 0 (\Om) \Big \}  \geq \frac{1}{4 N ^ { N+1} }\,.   \end{equation}
\end{lemma}
\proof Let $C_\Om$ denote the infimum at the left hand side in \eqref{f:tconstant}.   For $\Om = B _ 1$, it holds  
\begin{equation}\label{f:cb1} C _{B _ 1} = \frac{|B _ 1|}{|\partial B _ 1|}  = \frac{1}{N}\,,
\end{equation} see
 \cite[eq.\ (4)]{saintier}.  Let now  $\Om$ be a convex set as in the statement, let us assume without loss of generality 
  that the balls $B _1$ and $B _N$ are centred at the origin, and let 
 $u$ be a nonnegative function  in $ C ^ \infty (\R ^N )$. For every fixed $x \in \partial B _ 1$, denoting by $x_\Om$  the intersection point between $\partial \Om$ and the  halfline with origin at $0$ containing $x$, and by  $u'$ the derivative of the restriction of $u$ to the segment $[x ,  x_\Om]$, we have 
 $$u (x_\Om) = u ( x) + \int _{ [x ,  x_\Om]} u'\,.$$
Integration  over $\partial B _1$ yields
$$ \begin{aligned} \int _{\partial B _1} u (x_\Om)  & = \int _{\partial B _1}  u ( x) + \int _{\partial B _1}  \int 
_{[x ,  x_\Om] } u'
\\ \noalign{\medskip} 
& \leq \int _{\partial B _1}  u ( x) + \int _{\partial B _1}  \int _{ [x ,  x_\Om]} |\nabla u|\,.
\end{aligned}  
$$ 
 We now exploit 
\eqref{f:cb1} to estimate the first addendum in the last line, and we use polar coordinates to estimate the second one.  We obtain 
$$\begin{aligned} \int _{\partial B _1} u (x_\Om)  & \leq \frac{1}{N}     \Big (  \int _{B _1}  u ( x) + \int _{B _1}  |\nabla u| \Big )  +  \int _{\partial B _1}   \int _{ [x ,  x_\Om] }  \rho ^{N-1}  |\nabla u|   \\ 
\noalign{\medskip}
& = \frac{1}{N}     \Big (  \int _{B _1}  u ( x) + \int _{B _1}  |\nabla u| \Big )  +  \int _{\Om\setminus B _1}   |\nabla u| \\ 
\noalign{\medskip}
& \leq        \int _{\Om}  u ( x) + \int _{\Om}  |\nabla u| 
    \,. 
\end{aligned}
$$ 
To conclude the proof, we are going to show that
$$  \int _{\partial \Om}  u \leq 4 N ^ { N+1}  \int _{\partial B _1} u (x_\Om)\,. $$ 
By using Riemann approximations of the two involved surface integrals, this amounts to show that
\begin{equation}\label{f:element}
\limsup _{ \e \to 0} \frac{\mathcal H ^ { N-1} (\om _\e)}{ \mathcal H ^ {N-1} (\sigma _\e)} \leq 4 N ^ { N+1} \,,
\end{equation}
where  $\sigma _\e\subset \partial B _ 1$ is the intersection of $\partial B _1$ with a cone $\Gamma(0)$ with vertex at $0$ of opening angle infinitesimal as $\e \to 0$, and  
$\om _\e \subset \partial \Om$ is the image of $\sigma _\e$   through the map 
which sends any  $x  \in \partial B_1$ into the point $x_\Om$ defined as above. 

In order to prove \eqref{f:element},  
let us fix a system of coordinates with  $e_N$ along the axis of the cone $\Gamma (0)$. 
By this way, 
$\om _\e$ is locally described as the graph of a map $x_N = \varphi ( x')$.  

We consider a point $x^*  \in \om _\e$   
such that 
$$\| x _\Om \| \leq \| x ^* \| \qquad \forall x \in \sigma _\e\,,$$ 
and we denote by $\Gamma (x ^*)$ the cone obtained as the convex envelope of  $B _ 1$  and $x ^*$.
Let  
$$ \overline{x } _N := \inf \{ x_N \ :\ (x', x_N) \in \Gamma (x ^*)\cap \Gamma(0)   \}\,.$$ 

(The dependence on $\e$ of   $x^*$ and $\overline x_N$ is omitted for simplicity of notation, as well as for the cones
$\Gamma(0)$ and $\Gamma (x ^*)$).

By monotonicity under inclusion of the surface area measure of convex bodies, we have
\begin{equation}\label{f:mon} 
\mathcal H ^ { N-1} (\om _\e) \leq \mathcal H ^ { N-1} (\partial U _\e) \,,
\end{equation} 
where $U _\e$ is the convex set 
$$U _\e:= \Big \{ (x', x_N) \ :\  (x', \varphi (x') ) \in \om _ \e\, , \ \overline x_N  \leq x_N \, , \ \| (x', x_N) \| \leq \| x^* \|  \Big \} \,. $$ 
The boundary of $U _\e$ can be decomposed as
$\partial U _\e  = \partial  U _\e  ^{(1)}  \cup \partial  U _\e  ^{(2) }  \cup  \partial  U _\e  ^{(3) } $, with 
$$\begin{aligned}
& \partial  U _\e  ^{(1)}  : = \big \{ (x', x_N)  \in \partial U _ \e \ :\ x_N = \overline x _N \big \} 
\\ \noalign{\medskip} 
& \partial  U _\e  ^{(2)}  : = \big \{ (x', x_N)  \in \partial U _ \e \ :\ \| (x', x_N) \| = \| x^* \|  \big \} 
\\ \noalign{\medskip} 
& \partial  U _\e  ^{(3)}  : = \big \{ (x', x_N)  \in \partial U _ \e \ :\   (x', \varphi (x')) \in \partial \om _ \e \big \} \,.
 \end{aligned}
 $$ 
Since $\Om \subset B _N$, we have 
\begin{equation}\label{f:12} \mathcal H ^ { N-1} \big ( \partial  U _\e  ^{(i)} \big ) \leq N ^ { N-1} \mathcal H ^ { N-1}   (\sigma _\e ) \qquad \text{ for } i = 1, 2 \,.
\end{equation} 
On the other hand, the surface measure of the ``lateral''  portion $ \partial  U _\e  ^{(3)} $ of the boundary 
satisfies the estimate 
\begin{equation}\label{f:3prima} \mathcal H ^ { N-1} \big ( \partial  U _\e  ^{(3)} \big ) \leq N ^ { N-2} \mathcal H ^ { N-2}   (\partial \sigma _\e ) h_\e  \,,
\end{equation}   
where $h_\e$ is the ``height'' of $U _\e$, namely the length of its intersection with the $e_N$-axis. In order to give an upper bound for $h_\e$, we look at a section by a plane $\Pi$ containing the axis $e_N$.
 We have that $\Pi \cap  \{ \|x\| = \| x ^* \| \}$ is a circumference,  $\Pi \cap \Gamma (0)$ is a circular sector, and the intersection between 
$\Pi \cap  \{ \|x\| = \| x ^* \| \}$ and $\partial \Gamma (0)$ consists of two points.  We fix one of these two points, and we name it $P$. Still in the plane $\Pi$, 
among the two tangent lines to 
$B _1$ passing through $P$, we choose the one  which meets $\Gamma(0)$, and we denote by $P'$ the tangency point of such tangent line in $\partial B _1$. 
We set for brevity $\ell:= \| x ^*\|$, we denote by $2 \e$ the opening angle of the sector $\Pi \cap \Gamma (0)$
and  we let $\theta:= \arcsin \big ( \frac {1}{ \ell} \big )$. 
Then, if we fix in  $\Pi$  a system of cartesian coordinates $(X, Y)$ with $Y$ aligned with $e_N$, 
$P$ and $P'$ have coordinates
$$P= ( \ell \sin \e, \ell \cos \e)\, , \qquad P'= ( - \cos \theta, \sin \theta )\,,$$
see Figure \ref{fig:1}. 

\begin{figure} [h] 
\centering   
\def\svgwidth{7cm}   
\begingroup%
  \makeatletter%
  \providecommand\color[2][]{%
    \errmessage{(Inkscape) Color is used for the text in Inkscape, but the package 'color.sty' is not loaded}%
    \renewcommand\color[2][]{}%
  }%
  \providecommand\transparent[1]{%
    \errmessage{(Inkscape) Transparency is used (non-zero) for the text in Inkscape, but the package 'transparent.sty' is not loaded}%
    \renewcommand\transparent[1]{}%
  }%
  \providecommand\rotatebox[2]{#2}%
  \ifx\svgwidth\undefined%
    \setlength{\unitlength}{265.53bp}%
    \ifx\svgscale\undefined%
      \relax%
    \else%
      \setlength{\unitlength}{\unitlength * \real{\svgscale}}%
    \fi%
  \else%
    \setlength{\unitlength}{\svgwidth}%
  \fi%
  \global\let\svgwidth\undefined%
  \global\let\svgscale\undefined%
  \makeatother%
  \begin{picture}(6,1)%
    \put(0.6,0.1){
     \includegraphics[height=6.5cm]{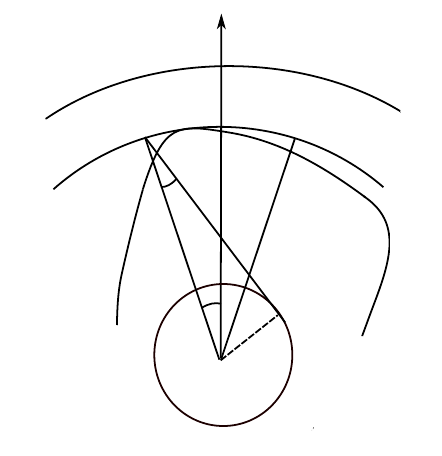}} %
    \put(1.45, 0.78){\color[rgb]{0,0,0}\makebox(0,0)[lb]{\smash{$\partial B _N$}}}%
    \put(1.42, 0.62){\color[rgb]{0,0,0}\makebox(0,0)[lb]{\smash{$\partial B _{ \|x^*\| } $}}}%
    \put(1.42, 0.5){\color[rgb]{0,0,0}\makebox(0,0)[lb]{\smash{$\partial \Omega$}}}%
    \put(1.22, 0.3){\color[rgb]{0,0,0}\makebox(0,0)[lb]{\smash{$\partial B _1$}}}%
    \put(0.97, 0.785){\color[rgb]{0,0,0}\makebox(0,0)[lb]{\smash{$x^*$}}}%
    \put(0.885, 0.76){\color[rgb]{0,0,0}\makebox(0,0)[lb]{\smash{$P$}}}%
    \put(1.195, 0.385){\color[rgb]{0,0,0}\makebox(0,0)[lb]{\smash{$P'$}}}%
    \put(1.135, 0.46){\color[rgb]{0,0,0}\makebox(0,0)[lb]{\smash{$P''$}}}%
    \put(0.965, 0.61){\color[rgb]{0,0,0}\makebox(0,0)[lb]{\smash{$\theta$}}}%
    \put(1.03, 0.415){\color[rgb]{0,0,0}\makebox(0,0)[lb]{\smash{$\e$}}}%
 \end{picture}%
\endgroup%
\null\vskip -1 cm 
\caption{The geometric argument in the proof of Lemma \ref{l:trace}.}
\label{fig:1}   
\end{figure}
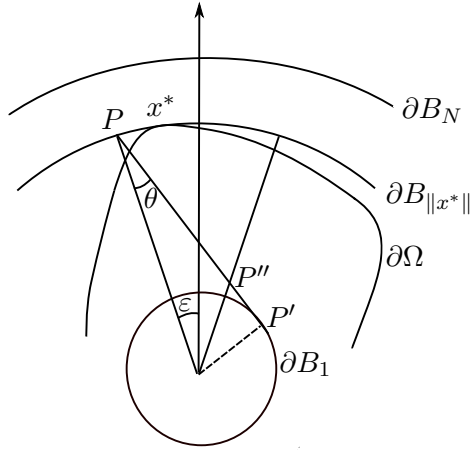

 We observe that, 
if  $P''= (P''_X, P'' _Y)$ denotes the intersection point between the line segment $[P, P']$ and 
$\Pi \cap \partial \Gamma (0)$ (i.e., the intersection of $[P, P']$ with the straight line $X = - (\tan \e) Y$),   $h _\e$  admits the upper bound
\begin{equation}\label{f:boundh}
h _\e \leq \ell - P'' _Y\,.
\end{equation}
In order to determine $P''_Y$, it is enough to solve the linear system 
$$\begin{cases} 
X = - Y \tan \e  & \\ 
\noalign{\bigskip}
\displaystyle \frac{ X + \cos \theta  }{Y - \sin \theta} = \frac{ \ell \sin \e + \cos \theta} {\ell \cos \e - \sin \theta } \,,
\end{cases}
$$ 
the last equation being a consequence of the fact that $P,P'$, and $P''$ are aligned. 
Then a straightforward computation gives, in the limit as $\e \to 0$, 
\begin{equation}\label{f:p''}
P'' _Y = \ell - 2 \frac{\cos \theta}{\sin ^ 2 \theta } \e + o ( \e) \,.
\end{equation}
From \eqref{f:boundh} and \eqref{f:p''}, and recalling that $\ell = \frac{1}{\sin \theta}\leq N$, we obtain
$$ h _\e \leq 2 \frac{\cos \theta}{\sin ^ 2 \theta } \e + o ( \e) \leq2 N ^ 2 \e + o ( \e)\,, $$ 
which combined with \eqref{f:3prima} gives 
\begin{equation}\label{f:3seconda}
\begin{aligned}
\mathcal H ^ { N-1} \big ( \partial  U _\e  ^{(3)} \big ) & \leq N ^ { N-2} \mathcal H ^ { N-2}   (\partial \sigma _\e ) \cdot 2 N ^ 2 \e + o ( \e)  \,,
\\ 
& =  2  N ^ { N}  (N-1) \mathcal H ^ { N-1}   ( \sigma _\e )  + o ( \e)  \,,
 \end{aligned}
\end{equation} 
In view of \eqref{f:mon}, \eqref{f:12}, and \eqref{f:3seconda}, we finally have
$$\limsup _{ \e \to 0} \frac{\mathcal H ^ { N-1} (\om _\e)}{ \mathcal H ^ {N-1} (\sigma _\e)} \leq 
2N ^ { N-1}  + 2 N ^ N (N-1) \leq 4 N ^ { N+1} \,. 
$$ 
\qed

\bigskip

	{\bf Proof of Proposition \ref{ambu25}}. 
	  Let $a_1 \geq a_2 \geq \dots \geq a_N $ be the John semi-axes of  $\Om$, and assume without loss of generality that 
	$a_1= 1$. 
	Set  $A: = T (\Om)$, where  
	$T(x)= X$ is given by 
	$$
	X_1=x_1, \quad X_2= \frac{x_2}{a_2}  , \quad \dots\quad,   \quad X_N= \frac{x_N}{a_N} \,,
	$$
and let	$v$ be defined on $A$ by 
		$$
	v(X_1,\dots , X_N) := \overline u(X_1, a_2  X_2, \dots , a_N X_N )
\,.$$

Setting $c = (1,  c_2, \dots, c _N)  :=  (1,  \frac{1}{a_2^2}, \dots, \frac{1}{a_N^2} ) $, and
$$|\nabla v| _ c:=   \Big ( \sum _{i = 1} ^N c_i \big | \frac{\partial v}{\partial X_i} \big | ^ 2 \Big ) ^ { \frac 1 2}\,.$$  
it holds 
	$$ \mu _ p (\Om):= 
	\frac{{\int_{\Omega} \abs{\nabla \overline u}^p }}{{\int_{\Omega}  \overline u^p} }= 
	\frac{{\int_{A} \abs{\nabla  v}_c ^p }}{{\int_{A}   v^p} } = : \mu _ {p, c} ( A)  \,.
	$$
	Since $a_1$ is comparable to $D_\Om$, in terms of  $v$  the inequality \eqref{ambu40} reads		\begin{equation}
		\label{intermedio.1}
		\norma{v}_{L^\infty(A)} \leq C \mu_{p, c} (A) ^{ \frac { N}{p}} \norma{v}_{L^p(A)} \,.
	\end{equation} 
	 To prove \eqref{intermedio.1}, we apply Moser iteration scheme. Using the weak form of the equation satisfied by $\overline u$, $v$  satisfies
	\begin{equation}\label{f:weakv} \sum_{i= 1} ^N  \int_{A}   |\nabla v| _ c ^ { p-2} c_i \frac{\partial v}{\partial X_i}\frac{\partial \vphi}{\partial X_i}= \mu_{p,c} (A) \int_A  |v| ^ { p-2} v \vphi \qquad  \forall \vphi \in W^{1, p}(A) \,.
	\end{equation} 
	 We claim that, 
	if a solution $v$ to \eqref{f:weakv} belongs to $L^q(A)$ for some $q \geq p$, then it belongs also to $L^{\frac{q N}{N-1}}(A)$;  more precisely, 
for a constant   $C_1= C_ 1(N, p)$ which can be explicitly computed  (see eq.\eqref{f:C1} below)    it holds 
			\begin{equation}\label{f:moser} 
	\Big( \int_{A} |v|^{\frac{
  qN  }{N-1}}\Big )^{\frac{N-1}{N}} \leq 
	 C  _1   \mu_{p,c} (A) 
	\frac{ (\frac{q}{p} ) ^p}{q-p+1}	\int_{A} |v|^q\,. 
	\end{equation} 
Indeed, if $v \in L ^ q (A)$, we can  take 	
	$\varphi = |v|^{q-p}v$ as test function in \eqref{f:weakv} for any $q \geq p$. We obtain 
	$$
	\sum_{i=1}^N \int_{A} |\nabla v| _ c ^ { p-2} c_i \frac{\partial v}{\partial X_i}\frac{\partial (|v|^{q-p}v)}{\partial X_i} = \mu_{p,c} (A) \int_A |v|^q\,,
	$$
By direct computation, the above equality can be rewritten as
$$(q-p+1) \int_A  |\nabla v| _ c ^ { p } \,  |v|^{q-p} = \mu_{p,c} (A) \int_A |v|^q\,,
	$$
 or equivalently as 
$$ \frac{ (q-p+1)  }{ \big ( \frac q p \big ) ^ p } \int_A  |\nabla   |  v ^ {\frac q p}     |   | _ c ^ { p }  = \mu_{p,c} (A) \int_A |v|^q\,.
	$$ 
From the inequality 
$c_i \geq 1$, we infer that
	$$
	\int_A \Big | \nabla |v|^{\frac{q}{p}}  \Big | ^p  \leq 	\mu_{p,c } (A) \frac{ \big ( \frac q p \big ) ^ p }  { (q-p+1)  }  \int_A |v|^q.
	$$
	
Now, for every positive constant $\alpha$, we have 
$$
\begin{aligned} 
		  \int_A \abs{\nabla |v|^{q}}&= 	\int_A  \big | \nabla (|v|^{\frac{q}{p}})^p\big | 
		= 	p \int _A |v|^{\frac{q(p-1)}{p} } \alpha ^ {\frac1 p}  \big | \nabla |v|^{\frac{q}{p}}\big | 
		\frac{1}{ \alpha ^ {\frac1 p} } 
		\\
		& \leq 	p  \Big [ \int _A \frac{1}{\frac{p}{p-1}}  |v|^{q } \alpha ^ {\frac{1}{ p-1}} + \int _A   \frac{1}{p}  \big | \nabla |v|^{\frac{q}{p}}\big |  ^p 
		\frac{1}{ \alpha }\Big ] \\
		& =	(p-1)    \alpha ^ {\frac{1}{ p-1}}   \int _A   |v|^{q } + \frac{1}{ \alpha } 
 \int _A   \big | \nabla |v|^{\frac{q}{p}}\big |  ^p 
		 		 \,.
	\end{aligned}
	$$ 
	
	We infer that 
				\begin{equation}\label{f:alfa1} 
	\int_{A} |v|^q +  \int_A \abs{\nabla |v|^{q}}  \leq 
		 	   \Big (1+ (p-1)    \alpha ^ {\frac{1}{ p-1}}   +  \frac 1\alpha \mu_{p,c } (A) \frac{ \big ( \frac q p \big ) ^ p }  { (q-p+1)  }  \Big ) 	\int_{A} |v|^ {  q  }  \,.  
\end{equation}

By the continuity of the embedding of $BV ( \R ^N)$ into $L ^ {\frac{N}{N-1}} (\R^N)$,   and  of the trace operator from $W ^{1, 1} (A)$ to $L ^ 1 (\partial A)$, 
 	for every  $w \in W ^ { 1, 1} ( A)$ (extended to zero outside $\overline A$) we have 
	\begin{equation}
		\label{sobol} \Big ( \int_A  \abs{w}^{\frac{N}{N-1}} \Big ) ^{\frac{N-1}{N}} \leq    C_2   \Big (  \int_A \abs{\nabla w} + \int_{\partial A} \abs{w} \Big )  \leq   C_3   \Big  ( \int_A \abs{\nabla w} +\int_A \abs{ w} \Big )\,.\end{equation}
		where  the constants  $C_2$ and $C_3$   are  purely dimensional (because 
$A$ has  outradius and  inradius  controlled respectively from above and from below). 
  Specifically, denoting by $\om _N$ the Lebesgue measure of the unit ball in $\R ^N$,  and using Lemma \ref{l:trace}
we have
$$C_2 = N ^ {-1}  \om _N ^ {-\frac{1}{N} } \, , \qquad C_3 =  C_2 (1+ 4 N ^ { N +1} ) = N ^ {-1}  \om _N ^ {-\frac{1}{N} }  (1+  4 N ^ { N +1} )   \,. $$

		Thus, applying \eqref{sobol}  with $w = |v|^q$,   we obtain 
		$$
	\Big ( \int_{A} |v|^{\frac{q N}{N-1}}\Big )^{\frac{N-1}{N}} \leq  C _3   \Big (1+ (p-1)    \alpha ^ {\frac{1}{ p-1}}   +  \frac 1\alpha \mu_{p,c } (A) \frac{ \big ( \frac q p \big ) ^ p }  { (q-p+1)  }  \Big ) 	 	\int_{A} |v|^q\,.
	$$
By the Payne-Weinberger type inequality for the first Neumann eigenvalue of the $p$-Laplace operator, recalling that   $a_ 1= 1$,  so that $D_\Om\le 2N$, we have  $$\mu_{p,c} (A)=  \mu _ p (\Om) \geq   ( \frac{\pi_p}{D_\Om}  ) ^ p  \geq( \frac{\pi_p}{2N}  ) ^ p  \,.$$ 
 Thus,  for  $\alpha$ such that   \begin{equation}\label{f:alfa}
 1+ (p-1)    \alpha ^ {\frac{1}{ p-1}}   \leq \frac 1\alpha ( \frac{\pi_p}{2N}   ) ^ p  \frac{ \big ( \frac q p \big ) ^ p }  { (q-p+1)  }\,,
 \end{equation}  
we have 
 $$1+ (p-1)    \alpha ^ {\frac{1}{ p-1}}   \leq \frac 1\alpha \mu_{p,c } (A) \frac{ \big ( \frac q p \big ) ^ p }  { (q-p+1)  }   \,,$$
	so that we shall obtain the inequality \eqref{f:moser}, with  $ C_1= \frac{ 2C_3} {\alpha}$.
	 
Let us determine for which values of $\alpha$ the inequality \eqref{f:alfa} is satisfied. 
We claim that this occurs provided
$$\alpha \leq \alpha _{N,p}:= \min \Big \{ \frac{1}{(2N) ^p}  , \Big ( \frac{1}{p-1} ( ({\pi_p}  ) ^ p -1) \Big  )^  {p-1} \Big \} \,.$$

Indeed, for such values of $\alpha$ we have
$$\alpha ^ {\frac{1}{ p-1}} \leq \Big ( \frac{1}{p-1} ( ({\pi_p}   ) ^ p -1) \Big   )$$ 
and hence also 
$$  1+ (p-1) \alpha ^ {\frac{1}{ p-1}} \leq  ({\pi_p}   ) ^ p \leq  \frac{1}{\alpha} \big (\frac{\pi_p}{2N}   \big  ) ^ p  \leq \frac 1\alpha \big ( \frac{\pi_p}{2N}   \big  ) ^ p  \frac{ \big ( \frac q p \big ) ^ p }  { (q-p+1)  } \,,$$
where the last inequality is due to the fact that the function   $\psi ( q):=\frac{ \big ( \frac q p \big ) ^ p }  { (q-p+1)  }  $ is monotone increasing for $q \geq p$.

We next observe that  $\alpha _{N, p} = \frac{1}{(2N) ^p}$. Indeed we claim that, for every $p  >  1$, there holds:   
$$ \Big ( \frac{1}{p-1} ( ({\pi_p}  ) ^ p -1) \Big  )^  {p-1}   \geq 1\,. $$
This is equivalent to $\pi_p^p \ge p$ or $\pi_p\ge p^\frac 1p$, or
 $$ 2\pi \frac{(p-1) ^ {\frac 1 p}}{p \sin  \big (\frac \pi p \big ) } \ge p^\frac 1p.
$$
 It is easy to check that $p^\frac 1p \le e^\frac 1e$ from the   piecewise monotonicity of $x \to \frac{\ln(x)}{x}$ on $(1, +\infty)$ so that it is enough to prove
 $$\frac{2\pi}{e^\frac 1e} (p-1) ^ {\frac 1 p}\ge p  \sin  \big (\frac \pi p \big ).$$
 We distinguish two cases. If $p \in (1,2]$ we have, using the inequality $\sin(x) \le x$ for $x\ge 0$, 
 $$p  \sin  \big (\frac \pi p \big ) = p  \sin  \big (\pi -\frac \pi p \big )\le p\big (\pi -\frac \pi p \big )= \pi(p-1),$$
 so it is enough to prove
 $$\frac{2\pi}{e^\frac 1e} (p-1) ^ {\frac 1 p} \ge \pi(p-1).$$
Noticing that   $\frac{2\pi}{e^\frac 1e} \approx 4.34 >\pi$, the inequality is
 true   for $p \in (1,2]$, since in this case we have $(p-1)^{\frac 1 p} \ge (p-1)$.
 
  If $p \in (2,+\infty)$ 
  we have
 $$p  \sin  \big (\frac \pi p \big ) \le p  \frac {\pi}{p}= \pi,$$
so it is enough to prove
 $$\frac{2\pi}{e^\frac 1e} (p-1) ^ {\frac 1 p} \ge \pi$$
 which is true since in this case $(p-1) ^ {\frac 1 p} \ge 1$.
 
 We conclude that  the inequality \eqref{f:moser} holds true with 
 \begin{equation}\label{f:C1} C_1= \frac{ 2C_3} {\frac{1}{(2N) ^p}} = 2 ^ { p+1} N ^ p C_3  = 2 ^ {p+1} N ^ p  \Big ( 
  \frac{1}{N} \om _N ^ {-\frac{1}{N} }  (1+  4 N ^ { N +1} ) \Big )    \,.
  \end{equation}

	We now apply such inequality recursively: setting $q _k=  p \big ( \frac{N}{N-1}  \big ) ^ k$, this gives
	$$
	\norma{v}_{L^{q_{k+1}} (A)} \leq  \Big (  C_1    \mu _ {p,c} (A)\Big )  ^{\frac{1}{q_k} }    
	\left(\frac{ (\frac{q_k}{p} ) ^p}{q_k-p+1}  \right)^{\frac{1}{q_k} }	\norma{v}_{L^{q_{k}}(A)}\,.
	$$
	In the limit as $k \to + \infty$ we obtain 
		$$
	\norma{v}_{L^\infty(A)} \leq \prod_{k=0}^\infty  \Big ( C_1    \mu_{p,c}(A) \Big )^{\frac{1}{q_k} } \left(\frac{ (\frac{q_k}{p} ) ^p}{q_k-p+1}  \right)^{\frac{1}{q_k} } \norma{v}_{L^p(A)}.
	$$
	
By noticing that 
		$$
	\sum_{k \geq 0} \frac{1}{q_k} = \frac{1}{p} \sum_{k \geq 0} \left(\frac{N-1}{N}\right)^k= \frac{N}{p},
	$$
	and
	$$
	\beta (N, p)  := \prod_{k=0}^\infty\left(\frac{ (\frac{q_k}{p} ) ^p}{q_k-p+1}  \right)^{\frac{1}{q_k} }  
	\,,$$
  we  obtain that the inequality \eqref{intermedio.1} is satisfied by taking
	$$
	\begin{aligned}
	C=C(N,p)& =  C_1    ^ \frac{N}{p}  \beta (N, p)
	\\ 
	& =  2 ^  {  (\frac{p+1}{p} ) N } N ^N     \big ( \frac{1}{N} \om _N ^ {-\frac{1}{N} } (1+ 4 N ^ { N +1} ) \big )   ^ { \frac{N}{p}  }  \beta (N, p)
	 \,.
	\end{aligned} 
		$$ 

Finally we observe that, since $q_k-p+1 \geq 1$, we have, writing for brevity $\delta:=  \frac{N}{N-1}$ 
$$\beta (N, p) \leq \prod_{k=0}^\infty\big({ \frac{q_k}{p} }  \big)^{\frac{p}{q_k} }   =  
\prod_{k=0}^\infty\big({ \delta ^k }  \big)^{ ( \frac{1}{\delta } ) ^k }  = \delta ^ { \sum _{k \geq 0} \frac{k}{ \delta ^k}} = \delta ^ {N(N-1)} = 
\big ( \frac{N}{N-1} \big  ) ^{N (N-1)}  \,. $$ 

Taking also into account that $\frac{p+1}{p} \leq 2$, we conclude that $C (N,p)$ is not larger than the constant $ \caone$ defined in \eqref{f:caone}.

\qed

\bigskip

	{\bf Proof of Corollary \ref{c:collapse}}.  We consider the sequence of open bounded convex domains in $\R ^ { N+m}$ given by 
	$$\Om _\e:= \Big \{ (x, y) \in  \Om  \times \R ^m \ :\ x \in \Om   \, , \     \| y \| _{\R ^m}  <   \e    \om _m ^ { -\frac{1}{m} }  
	 \phi ^{ \frac{1}{m} } (x) \Big \}  \,.$$ 
	 The inequality \eqref{f:inftybound} follows by applying Proposition \ref{ambu25} to the sets $\Om _\e$ and passing to the limit as $\e \to 0$ according to Proposition \ref{abf209}. 
	 \qed

	\section{Proof of Theorem \ref{t:Ngap}}\label{sec:proof1}
	
 	We proceed as follows: in Section \ref{s:s1} we prove  the  new tool
with respect to the linear case, which is a refinement of a one-dimensional weighted nonlinear Poincar\'e inequality, inspired from \cite{ENT, FNT12}; in Section \ref{s:s2} we introduce the concept of modified Payne-Weinberger partitions, which is adapted from \cite{ABF24}, and we prove the rigidity of the inequality \eqref{paynewein}; 
in Section \ref{s:s3} we give an estimate for the Lebesgue measure of the cell of the partitions, which is obtained by exploiting the $L ^ \infty$ estimates of Section \ref{sec:infty}; finally in Section \ref{s:s4} we prove Theorem \ref{t:Ngap}.

 We shall need the observation that 
$$ \lim_{p \to 1_+}    (\pi _p)   ^ p=2,  \quad \lim_{p \to +\infty}  (\pi _p)  ^ p=+\infty, \quad 2< (\pi_p) ^p<+\infty.$$
While  the computation of the limits is elementary, to prove the lower bound $2$ for $(\pi _p )  ^p$ one has 
 to split the discussion in the  two cases $p\ge 2$ and $1<p<2$. Using respectively 
 $\sin\big (\frac \pi p \big ) < \frac{\pi}{p} $ 
 and $\sin\big (\frac \pi p \big ) = \sin\big (\frac {(p-1)\pi}{ p} \big ) < \frac {(p-1)\pi}{ p}$, one gets 
 $$
 \begin{cases} 
\pi_p^p = \Big (2 \pi \frac{(p-1) ^ {\frac 1 p}}{p \sin  \big (\frac \pi p \big ) } \Big)^p\ge \Big (2 \pi \frac{(p-1) ^ {\frac 1 p}}{p   \big (\frac \pi p \big ) } \Big)^p= 2^p (p-1) > 2
& \text { if } p \ge 2 
\\ \noalign{\medskip} 
\pi_p^p = \Big (2 \pi \frac{(p-1) ^ {\frac 1 p}}{p \sin  \big (\frac \pi p \big ) } \Big)^p\ge \Big (2 \pi \frac{(p-1) ^ {\frac 1 p}}{p   \frac {(p-1)\pi}{ p}   } \Big)^p= 2 \Big (\frac{2}{p-1}\Big)^ {p-1} > 2.
& \text { if } 1< p < 2  \,.
\end{cases}
$$

  \subsection{One dimensional refined version of nonlinear Poincar\'e inequality}\label{s:s1} 
By  Payne-Weinberger reduction argument,  
the key point in order to estimate  from below $\mu _ p (\Om, \phi)$   for a domain $\Om \subset \R ^N$, say  of diameter $1$,   amounts to estimate from below  a weigthed one-dimensional Neumann eigenvalue of the type
 $\mu _ p (I_d, h \phi)$,  
where $I_d$ is a line segment  of length $d \in (0, 1]$ contained into $\Om$,  
$\phi$ is the preassigned power-concave weight, and $h$ is the  power-concave function giving 
the $(N-1)$-dimensional measure of the cell's section orthogonal to    $I_d$.   
Since $h \phi $ is log-concave,  we have that 
$$ \mu _p (I_d, h \phi) \geq  \frac{ \big ({\pi_p}  \big ) ^ p}{d^p} \,.   $$

\medskip
In Proposition \ref{p:1d} below, we are going to provide a refined version of the above inequality. 
In order to express the involved constants, let us define   
 \begin{equation}
    \label{f:cafour}  \cafour  =  \cafour (p, m):=   \caone (\pi _ p ^ p + 1 )^{\frac{ m +1 }{p}} \frac{1}{{\left (\int_{I} 
( \min \{x , 1 - x \} ^ {m+1} ) \right )}^{\frac 1 p} } \,,   
\end{equation} 
  where $\caone$  denotes  the value  of  the constant  defined in  \eqref{f:caone} at $m+1$ and $I=I_1$.  
 
Moreover, let us introduce 
two auxiliary functions of four variables $(p, a , s,t)$, with  $p>1$, $a\in (0, 1)$, and $s,t >0$,   respectively by
 \begin{equation}\label{f:b0}
b _ 0 = b _0 ( p, a, s , t) := \min \Big \{ \frac{a}{2},  \Big ( \frac{1}{2} \frac {p-1}{p} \frac{s}  { t}  \Big )  ^ {\frac{p}{p-1}  }  \Big \} \,,
\end{equation}    
and 
 \begin{equation}\label{f:M} M = M ( p, a, s, t) := \begin{cases}
  \Big ( \frac{p-1}{p}  \Big ) ^ { p-1} \cdot \min \Big \{ \frac{b _0 ^2}{2}, \frac{b_0}{2}    s ^{  \frac{p}{p-1} }  \Big \}  & \text{ if } p \in (1, 2) \, , 
\\ \noalign{\medskip} 
\frac{1}{2} \Big ( \frac{p-1}{p}  \Big ) ^ { p-1}  
 b _ 0 ^ { p-1}   s  ^ p  & \text{ if } p \geq 2 \,.
\end{cases}
\end{equation}     
For $\cafour(p, m)$ and $M (p, a , s, t)$ defined as above, set

\begin{equation}\label{f:cafive}
\cafive  = \cafive (p,m):=  \,     \frac{1}{2 ^ {\frac{1}{p-1}}} 
  \frac{e ^  { -  (m+1)  2  ^ {m+3}  \cafour ^ { p(m+2)}    }   }{  \cafour  ^  {\frac{2p^2-p}{p-1} }} 
M  \Big (   p,  \frac{1}{4 \cafour ^ p } ,  \frac{1}{(4 \cafour ) ^ {\frac{1}{p-1}} }\,  ,
 ( 4 \cafour  ^ p) ^ { m+1}   \big ( (\pi _p ) ^ p + 1 \big ) \Big ) 
 \end{equation}
 \begin{equation} \label{f:casix} 
\casix =  \casix (p, m):=  \frac{1}{4 \cafour ^ p }  \,. 
\end{equation}

\begin{proposition}\label{p:1d}  Let $\cafive = \cafive (p, m)$ and $\casix = \casix (p, m)$ be defined by \eqref{f:cafive} and \eqref{f:casix}. For any function  $ f$ on $I= (0, 1)$ such that $\| f \| _\infty = 1$ and $f=h \phi$, being $h$  positive concave and $\phi$   positive $(\frac{1}{m})$-concave,
it holds 
\begin{equation}\label{f:1dq}
\mu _ p ( I , f) \leq    (  {\pi _p} ) ^ p   +1 \ \Rightarrow \ \mu _ 1 ( I , f)  \geq   (  {\pi _p} ) ^ p   +  \cafive  \min _{[\casix, 1 - \casix]} \Big ( \frac{h'}{h} \Big ) ^ 2 \,.
\end{equation}
\end{proposition} 

  In the next two remarks, we reformulate the inequality \eqref{f:1dq} first in a rescaled version on the interval $I _d= (0, d)$ with $d <1$ (see \eqref{f:1dqbis}), 
and then  in a ready-to use version where the dependence on $d$ of the involved constants is weakened (see \eqref{f:1dqbis.1}).

\begin{remark} The above inequality scales as follows. If $I_d=(0,d)$ and $f=h \phi$, with $h, \phi:I_d\to \R^+$ concave, $(\frac{1}{m})$-concave, 
respectively,   by applying   to  the weight   $F(x)=f(d\cdot x)$, 
one gets 
\begin{equation}\label{f:1dqbis} 
\mu_p(I_d, f) \leq   \Big( \frac{ {\pi _p} }{d}\Big) ^ p    +\frac{1}{d^p}  \ \Rightarrow \ \mu _ 1 ( I_d , f)  \geq   \Big( \frac{ {\pi _p} }{d}\Big) ^ p   +  \cafive d^{2-p} \min _{[d\casix , d - d\casix]} \Big ( \frac{h'}{h} \Big ) ^ 2 \,.
\end{equation} 
\end{remark}

\begin{remark} \label{r:pip} 
For every $p$, let  $d_p$ be such that $ \Big( \frac{ {\pi _p} }{d_p}\Big) ^ p =     {\pi _p}  ^ p   +1$, that is, 
\begin{equation}\label{f:dp} 
d_p:= 
\frac{  \pi _p   }{  \big ( (\pi _p )   ^ p +1 \big ) ^ {\frac 1p }  } .
\end{equation} 
For $d <d_p$, we have $\mu_p(I_d, f) >\Big( \frac{ {\pi _p} }{d_p}\Big) ^ p =     ( {\pi _p}  ) ^ p   +1$. 
  For $d \in [ d _p, 1]$,     from  
\eqref{f:1dqbis} one gets the following inequality
\begin{equation}\label{f:1dqbis.1} 
  \mu_p(I_d, f) \leq   (  {\pi _p} ) ^ p    + 1   \ \Rightarrow \     \mu _ 1 ( I_d , f)  \geq     ({\pi _p} ) ^ p   +  \cafive\frac 23 \min _{[\casix , d - \casix]} \Big ( \frac{h'}{h} \Big ) ^ 2. \end{equation} 
 Indeed, since  for $d \in [ d _p, 1]$ we have $\mu_p(I_d, f) \leq   \Big( \frac{ {\pi _p} }{d}\Big) ^ p    +\frac{1}{d^p}$, in order to pass from \eqref{f:1dqbis}  to \eqref{f:1dqbis.1}, it is enough to show that  $d^{2-p} \ge \frac 23$.   This is a consequence of 
 the inequality 
 $\frac{2}{3} < d _ p <1$ (which holds true because
$\pi_p^p >2$, so that $d_p = \big(\frac{\pi_p^p}{1+\pi_p^p}\big)^\frac 1p\ge \big( \frac 23\big)^\frac 1p >\frac 23$).  Indeed, if $p \ge 2$ the inequality  $d^{2-p} \ge \frac 23$ is trivial, while if $p \in (1,2)$ then $d^{2-p} \ge d_p^{2-p}\ge d_p > \frac 23 $.

\end{remark}

\medskip 
For the proof of Proposition \ref{p:1d}, we need the following.

\begin{lemma}\label{l:2cases}  Let $M   = M (\cdot, \cdot , \cdot, \cdot) $ be the function defined by \eqref{f:M}. 
Then,  for every function $g \in  W ^ { 1,p} (0, a)$ with $g ( 0)  = 0$ and $g' >0$ a.e.  on $(0, a)$, it holds  
$$ \int_0 ^ a x ( g' (x)  ) ^ { p-1} g ( x) \geq    M ( p, a, \| g \| _\infty, \| g' \| _p)   \, . $$ 
\end{lemma}

\proof We observe that 
$$  \int_0 ^ a x ( g'  ) ^ { p-1} g  =  \Big ( \frac{p-1}{p}  \Big ) ^ { p-1} \int_0 ^ a x  [ ( g ^ {\frac{p}{p-1}} ) '  ] ^ { p-1}  \geq  
\Big ( \frac{p-1}{p}  \Big ) ^ { p-1} \int_0 ^ a x   [ (( g ^ {\frac{p}{p-1}} )' )  ^* ]   ^ { p-1} 
\,.$$
where $[( g ^ {\frac{p}{p-1}} )' ]   ^*$ denotes the  
nonincreasing rearrangement  of $( g ^ {\frac{p}{p-1}} )'  $, and we have used Hardy-Littlewood inequality. 
If we define the function $\eta$  on $(0, a)$ by 
$$\eta (0) =  0 \, , \qquad \eta ' = [ ( g ^ {\frac{p}{p-1}} )' ]  ^*  \,,$$ 
we can rewrite the above inequality as 
\begin{equation}\label{f:formula1}  \int_0 ^ a x ( g'   ) ^ { p-1} g   \geq  
\Big ( \frac{p-1}{p}  \Big ) ^ { p-1} \int_0 ^ a x   ( \eta '  )   ^ { p-1} 
\,.\end{equation} 
We point out that $\eta$  enjoys the following properties: it is an increasing function, with $\eta ( a) = g(a) ^{  \frac{p}{p-1} }$, 
and it belongs to $W ^ { 1, p} (0, a)$, since
 $$\int _0 ^ a  ( \eta ' ) ^ p =  \int _0 ^ a [ ( g ^ {\frac{p}{p-1}} )'  ]   ^ { p}   = \Big ( \frac{p}{p-1}   \Big ) ^ { p}  \int _0 ^ a ( g')  ^ p g ^  {\frac{p}{p-1}  } 
\leq  \Big ( \frac{p}{p-1}   \Big ) ^ { p} \| g \| _\infty  ^  {\frac{p}{p-1}  }  \int _0 ^ a ( g')  ^ p \,.
$$ 
Using this estimate and H\"older inequality with conjugate exponents $p$ and $p'$, we also obtain that $\eta$ satisfies, for every 
$b \in ( 0, a)$,    
$$\eta ( b) = \int _0 ^ b \eta ' \leq b ^  { \frac{1}{p'} } \Big [  \int _0 ^ b (\eta ') ^ p   \Big ] ^  { \frac{1}{p}   }   \leq 
b ^  { \frac{1}{p'} }  \Big ( \frac{p}{p-1}   \Big )  \| g \| _\infty  ^  {\frac{1}{p-1}  } \| g' \| _p\,.
$$ 
Then we introduce $\overline b$ as the  positive solution to the equation 
$$\overline b ^  { \frac{1}{p'} }  \Big ( \frac{p}{p-1}   \Big )  \| g \| _\infty  ^  {\frac{1}{p-1}  } \| g' \| _p 
=  \frac {  g(a) ^{  \frac{p}{p-1} }   }{2} \big ( = \frac{\eta (a) }{2} \big ) \,, 
$$ 
namely
$$\overline b = \Big ( \frac{1}{2} \frac {p-1}{p} \frac{\| g \| _\infty} { \| g' \| _p}  \Big )  ^ {\frac{p}{p-1}  }\,.$$ 
We set $b _0: = \min  \big \{ \overline b,  \frac{a}{2} \big \} $ and we observe that, by construction, it holds
\begin{equation}\label{f:unmezzo} 
\eta ( b _0 ) \leq \frac{\eta (a)}{2}\,.  
\end{equation} 
We are now ready to conclude the proof, and we proceed  by distinguishing the two cases $p \in (1, 2)$ and $p \geq 2$.

\medskip 
{\it Case $p \in ( 1, 2)$}. We fix  $t \in [0, a]$   such that $\eta ' \geq 1$ on $(0, t)$ and $\eta ' \leq 1$ on $(t, a)$ (notice that this   can be done, 
possibly with one of the two intervals $(0, t)$ or $(t, a)$ empty),  since by construction 
$\eta '$ is nonincreasing).  Then, using the assumption $p  \in ( 1, 2)$, we can continue the estimate \eqref{f:formula1} as
$$\begin{aligned}
\int_0 ^ a x ( g'   ) ^ { p-1} g   & \geq  
\Big ( \frac{p-1}{p}  \Big ) ^ { p-1} \int_0 ^ a x   ( \eta '  )   ^ { p-1}   \geq 
\Big ( \frac{p-1}{p}  \Big ) ^ { p-1}  \Big [ \int_0 ^ t x + \int _ t ^ a x \eta ' \Big ] 
\\  
& = \Big ( \frac{p-1}{p}  \Big ) ^ { p-1}  \Big [ \frac{t ^ 2}{2}  +\int _ t ^ a x \eta '    \Big ]   
\end{aligned} 
$$ 
Now we distinguish two possible situations:
\begin{itemize}\item[--] If  $t \geq b _0 $, we drop the second addendum in last line  and we obtain
$$\int_0 ^ a x ( g'   ) ^ { p-1} g \geq  \Big ( \frac{p-1}{p}  \Big ) ^ { p-1}  \frac{b_0 ^ 2}{2} $$ 
\item[--] If  $b_0 \geq t  $, we drop the first addendum and by using \eqref{f:unmezzo} we obtain 
$$ \begin{aligned}
\int_0 ^ a x ( g'   ) ^ { p-1} g & \geq   \Big ( \frac{p-1}{p}  \Big ) ^ { p-1}  b _ 0 [ \eta (a) - \eta (b_0 )]  \geq    \Big ( \frac{p-1}{p}  \Big ) ^ { p-1} \frac{b_0}{2}    g(a) ^{  \frac{p}{p-1} } 
\\ 
&  = 
 \Big ( \frac{p-1}{p}  \Big ) ^ { p-1} \frac{b_0}{2}    \| g \| _\infty ^{  \frac{p}{p-1} }
\end{aligned} \,.
 $$ 
\end{itemize} 

\medskip 
{\it Case $p \geq 2$}.  In this case  we can continue the estimate \eqref{f:formula1} by using H\"older inequality with conjugate exponents $(p-1)$ and $(p-1)'$, and
then \eqref{f:unmezzo}. We obtain 
$$\begin{aligned}
\int_0 ^ a x ( g'   ) ^ { p-1} g   & \geq  
\Big ( \frac{p-1}{p}  \Big ) ^ { p-1} \int_0 ^ a x   ( \eta '  )   ^ { p-1}   \geq \Big ( \frac{p-1}{p}  \Big ) ^ { p-1}  
\frac {\Big [ \int_0 ^ a x   \eta '  \Big]    ^ { p-1}  } { \Big[  \int_0 ^ a x   \Big ] ^ {\frac{p-1}{(p-1)'}}  } 
 \\ 
 &=   \Big ( \frac{p-1}{p}  \Big ) ^ { p-1}  
\frac {\Big [ \int_0 ^ a x   \eta '  \Big]    ^ { p-1}  } { \big[ \frac{a ^ 2}{2}  \big ] ^ {p-2  } } 
\geq 
 2 ^ { p-2} \Big ( \frac{p-1}{p}  \Big ) ^ { p-1}  
 \frac{  b _ 0 ^ { p-1}  }{a ^ { 2 (p-2)}  } \big [ \eta (a) - \eta (b _0)   \big]    ^ { p-1}  
 \\ 
& \geq 
 2 ^ { p-2} \Big ( \frac{p-1}{p}  \Big ) ^ { p-1}  
 \frac{  b _ 0 ^ { p-1}  }{a ^ { 2 (p-2)}  } \big [\frac{1}{2}  \eta (a)    \big]    ^ { p-1}  
 = \frac{1}{2} \Big ( \frac{p-1}{p}  \Big ) ^ { p-1}  
 \frac{  b _ 0 ^ { p-1}  }{a ^ { 2 (p-2)}  }  \| g \| _\infty ^ p \\
 &   \geq
  \frac{1}{2} \Big ( \frac{p-1}{p}  \Big ) ^ { p-1}  
 {  b _ 0 ^ { p-1}  } \| g \| _\infty ^ p  
  \,.
  \end{aligned} 
$$

\bigskip

{\bf Proof of Proposition \ref{p:1d}}. By approximation, we can assume that the functions $h, \phi$ are smooth  with a strictly positive infimum on $\overline I $. 
In this case, an eigenfunction for  $\mu _ p ( I , f)$ is of class $\mathcal C ^ 1$ 
and  satisfies  
\begin{equation}\label{f:eigpb} 
\begin{cases} 
-  ( f u' |u'| ^ { p-2} )' = \mu _ { p   }( I , f) f  u |u| ^ { p-2}   & \text{ on } I 
\\  
u' (0) = u' (1) = 0 \,;
\end{cases}
\end{equation} 
For convenience, in the sequel we  normalize $u$ so that $\int _ I |u| ^ p f = 1$.  
We also have that $u'<0 $ on $I $,  see  \cite[proof of Lemma 2.4]{FNT12}, so that  $u$   has a unique zero on $I$, that we denote by $z$. 
  We divide the proof into several steps. 

\bigskip
{\bf Step 1. }  We have the lower bound 
\begin{equation}\label{f:primobound}
\mu _ p ( I, f) \geq  \big( {\pi _p} \big ) ^ p    + \frac{ \int _{I } e ^ { \kappa x}  | \psi - \kappa|     |v'| ^ { p-1}  |v| }{ \int _{I } e ^ { \kappa x} |v| ^ { p}} \end{equation}   
where 
$\kappa:= \frac{f' (z)}{f(z)}$,
\begin{equation}\label{f:defv} v (x):= \begin{cases}
 \frac{u ^- (x+z)}{u ^ - (1) } & \text{ for } x \in [0, 1 - z] 
 \\ 
 \frac{u ^+ ( x  - 1 + z) }{u ^+ (0) } & \text{ for } x \in [1 - z, 1]\,,
 \end{cases}
   \end{equation} 
and 
$$ \psi (x):= \begin{cases}
  (\log f)' (x+z)  & \text{ for } x \in [0, 1 - z] 
 \\ 
 (\log f)' (x- 1+z) & \text{ for } x \in [1 - z, 1]\,.
 \end{cases}
  $$

To obtain the lower bound \eqref{f:primobound}, we observe that the equation in \eqref{f:eigpb}  is satisfied pointwise by $u$ and can be rewritten as 
$$-   ( u' |u'| ^ { p-2} )' = \mu _ 1 ( I , f)   u |u| ^ { p-2}  + (\log  f)' u' |u'| ^ { p-2}  \qquad  \text{ on } I\,.$$

 By construction, $v$ is of class $\mathcal C ^ 1$ on $I $, with  $$v (0) = v ( 1 ) = 0 \, , \qquad  \lim _{ x \to (1 - z) ^ \pm }  v ( x) = 1\, ,  \qquad  \lim _{ x \to (1 - z) ^ \pm }  v' ( x) =   0\,.$$ 
 Therefore it satisfies the following equation in  $[ W ^ { 1, p} _0( I ) ]' $
 $$ \begin{aligned} -   ( v' |v'| ^ { p-2} )' & = \mu _ 1 ( I, f)  v |v| ^ { p-2}  + \psi  v' |v'| ^ { p-2}  
\\ 
& = \mu _ 1 ( I, f)   v |v| ^ { p-2}  + \kappa  v' |v'| ^ { p-2}  +  (\psi -\kappa)  v' |v'| ^ { p-2} 
\,, \end{aligned}$$ 
 Taking $e ^ { \kappa x} v$ as a test function gives
 $$ \int _{I }    e ^ { \kappa x} |v'| ^ p   =\mu _ 1 ( I, f)   \int _{I } e ^ { \kappa x} |v| ^ { p}  + 
 \int _{I } e ^ { \kappa x} (  \psi - \kappa)     v' |v'| ^ { p-2}  v \,.$$  
 Hence, 
 $$\mu _ p ( I , f)  = \frac{\int _{I }    e ^ { \kappa x} |v'| ^ p }{  \int _{I } e ^ { \kappa x} |v| ^ { p}} - \frac{ \int _{I } e ^ { \kappa x} (  \psi - \kappa)     v' |v'| ^ { p-2}  v }{ \int _{I } e ^ { \kappa x} |v| ^ { p}} 
 $$
 We observe that the functions $v'$ and $(\psi - \kappa)$ have opposite signs on $I $. Indeed,   on $(0, 1- z)$  we have $v' >0$ and $\psi (x) \leq \kappa = \psi (z)$ (from the log-concavity of $f$); similarly, 
on $(1 - z, 1)$ we have $v' <0$ and $\psi (x) \geq \kappa$.   
 Therefore we have
\begin{equation} \label{f:inspection} 
\begin{aligned}
\mu _ p ( I , f)  & = \frac{\int _{I }    e ^ { \kappa x} |v'| ^ p }{  \int _{I } e ^ { \kappa x} |v| ^ { p}} + \frac{ \int _{I } e ^ { \kappa x}  | \psi - \kappa|     |v'| ^ { p-1}  |v| }{ \int _{I } e ^ { \kappa x} |v| ^ { p}} 
\\ \noalign{\medskip} 
& \geq \big( {\pi _p} \big ) ^ p    + \frac{ \int _{I } e ^ { \kappa x}  | \psi - \kappa|     |v'| ^ { p-1}  |v| }{ \int _{I } e ^ { \kappa x} |v| ^ { p}}\,, 
\end{aligned} 
\end{equation}
where  the last inequality  has
 has been proved for $p \geq 2$ in  \cite{ENT} (see also \cite{FNT12}), 
while for $p \in (1, 2)$ in \cite{FNT12}.

\bigskip
 {\bf Step 2.} Let us prove the following estimates   (where the constant $\caone$ is computed at $(m+1)$)  
 \begin{eqnarray}
& \min \{x , 1 - x \} ^ {m+1}
  \leq  f  (x )  \leq 1  \qquad \forall x \in I\,, & \label{f:sstima3}  \\ 
   \noalign{\medskip} 
& - \frac{(m+1)}{1 -x} \leq f'  (x) \leq \frac{m+1}{x} \qquad \forall x \in I\,, & \label{f:sstima4}  \\ 
\noalign{\medskip}
&  	\norma{  u}_{\infty} \leq    \caone     (\pi _ p ^ p + 1 )^{\frac{ m +1 }{p}} \frac{1}{{\left (\int_{I} 
( \min \{x , 1 - x \} ^ {m+1} ) \right )}^{\frac 1 p}}  =: \cafour  \,,  & \label{f:sstima0} \\
\noalign{\medskip} 
& \min \{ z , 1 - z \}  \geq \frac{1}{ 2 \cafour ^ p } \,, & \label{f:sstima1}
 \\  \noalign{\medskip} 
& x \| u \| _\infty ^ {p-1} + |u ( x) | ^ {p-1} \geq \frac{1}{ 2 \|u \| _\infty} \qquad \forall x \in  I \,, &  \label{f:sstima2} 
\\  \noalign{\medskip} 
& |\kappa|    \leq      {(m+1)} { ( 2  \cafour ^p) ^ { m+2} }  =: \kappa ^*  \,.  &  \label{f:sstima5}  
\end{eqnarray}

The right inequality  in \eqref{f:sstima3} holds by assumption, while 
the left inequality is a straightforward consequence of the $({ \frac{1}{m+1}} )$-concavity of $f $, 
which follows from the fact that  
 $f ^ { \frac{1}{m+1}}$ is the geometric mean of the concave functions $h$ and $\phi ^  {{\frac{1}{m} }}$,
and hence it  is concave. 
Still as a consequence of the power concavity of $f $,   we have 
$$1 \geq f  ^ { \frac{1}{m+1}} ( x) -  f  ^ { \frac{1}{m+1}} ( 0) = \int _0 ^ { x}  
 \big (  f  ^ { \frac{1}{m+1}} \big ) ' \geq    x   \big (  f ^ { \frac{1}{m+1}} \big ) '  (x)   \,, $$ 
yielding  the upper bound for $f'  ( x)$ in \eqref{f:sstima4}.  The lower bound is obtained in the same way, integrating on the sub-interval $(x , 1)$.

To obtain \eqref{f:sstima0}, we recall from Corollary \ref{c:collapse} that 
$$
 	\norma{  u}_{\infty} \leq    \caone   \mu _p ( I , f)^{\frac{ m +1 }{p}} \frac{1}{{\left (\int_{I} 
f \right )}^{\frac 1 p}} \norma{u}_{L^p(\Om ,f )}\,,  
$$ 
 where the constant $\caone$ is computed at $m+1$. 

Then, by using the assumption $\mu _ p ( I , f) \leq    (  {\pi _p} ) ^ p   +1$, the normalization condition on $u$, and the lower bound on $f$ in \eqref{f:sstima3}, we get  
$$
 	\norma{  u}_{\infty} \leq     \caone    (\pi _ p ^ p + 1 )^{\frac{ m +1 }{p}} \frac{1}{{\left (\int_{I} 
( \min \{x , 1 - x \} ^ {m+1} ) \right )}^{\frac 1 p}}  =: \cafour.  
$$ 
To prove   \eqref{f:sstima1} we observe that 
$$ \begin{aligned}2 z \cafour  ^ p   & \geq  2 z \| u  \| _\infty ^ { p } \geq   2  \| u \| _\infty \int _{I}  | u  ^+ | ^ { p-1  } f  \\ &=     \| u  \| _\infty\Big [  \int _{I}  | u  ^+ | ^ { p-1  } f + \int _{I}  | u  ^- | ^ { p-1  } f  \Big ]  \\ &\geq     \int _{I}  | u  ^+ | ^ { p  } f + \int _{I}  | u ^- | ^ { p  } f    = \int _{I } |u | ^ p f = 1  \,,\end{aligned}  $$ 
where the first inequality holds because $(u )_+$  is supported on $[ 0, z]$ and $f  \leq 1$, while the second equality holds because  $\int _{I} f  |u | ^ { p-2} u  = 0$.
 The inequality  $2 (1-z) \cafour  ^ p   \geq 1$ is obtained in analogous way. 

Let us prove \eqref{f:sstima2}.  We observe firstly that
\begin{equation}\label{f:piumeno}
\int _I  |u ^\pm  | ^ { p-1} \geq \frac{1}{2 \| u \| _\infty }\,.
\end{equation}
Indeed, we have 
$$\int _I |u ^+|  ^ {p-2} u ^+ f = \int _I |u ^-|  ^ {p-2} u ^- f  \geq  \frac{1}{ \|u \| _\infty } \int _I  | u ^ -| ^ p f  =   \frac{1}{ \|u \| _\infty } \Big(  1- \int _I  | u ^ +| ^ p f \Big ) \,,$$   
 so that
 $$ \|u \| _\infty  \int _I |u ^+|  ^ {p-2} u ^+ f + \int _I  | u ^ +| ^ p f  \geq 1\,,$$
 which in view of $0 < f \leq 1$, 
implies \eqref{f:piumeno} for $u ^+$ (the case of $u ^-$ is analogous). 

Now, for every $x \in I$, we have

$$ \begin{aligned}
 x \| u \| _\infty ^ {p-1} + |u^ + (x) | ^ {p-1}   &\geq 
x \| u \| _\infty ^ {p-1} + |u^ + (x) | ^ {p-1}  (z-x) 
\\ \noalign{\medskip}
& \geq 
\int_0 ^ x   |   u ^ + (0) |   ^ {p-1} +  \int _ x ^ z   |   u ^ + (x)   |    ^ {p-1} \geq \int  _0 ^ z  |u ^+| ^ {p-1}  \geq \frac{1}{2 \| u \| _\infty }   \,,
\end{aligned} $$ 
where in the last equality we have used \eqref{f:piumeno}. We have thus proved \eqref{f:sstima2} in case $u ( x) = u ^ + (x)$.   The proof in case $u  (x) = u ^ - ( x)$ is analogous.

Finally the estimate \eqref{f:sstima5} readily follows from \eqref{f:sstima3},  \eqref{f:sstima4}, and \eqref{f:sstima1}.

\medskip 
{\bf Step 3.}   
We are going to prove that
\begin{equation} \label{abf.57}
\begin{aligned}  \frac{ \int _{I } e ^ { \kappa x}  | \psi - \kappa|     |v'| ^ { p-1}  |v| }{ \int _{I } e ^ { \kappa x} |v| ^ { p}}  \geq 
\frac{1}{2^ {\frac{p}{p-1}}  }  \frac{e ^ {-2 \kappa^*}    }{  \cafour  ^  {\frac{2p^2-p}{p-1} }}    \min _{[\delta, 1- \delta]} \Big ( \frac{h'}{h} \Big ) ^ 2 
    \cdot
     \Big   [  & \int _\delta ^ z  (z-y)    |(u^+) ' | ^ { p-1}  |u^+|  +  \\ 
     \noalign{\medskip} 
       &  \int _z ^  {1 - \delta}   (y-z)     |(u^-) '  | ^ { p-1}  |u^- |  \Big ]  \,.
       \end{aligned} 
       \end{equation}

    We first 
combine Step 1 and Step 2 to  obtain (from \eqref{f:sstima0} and \eqref{f:sstima2})  
$$ \begin{aligned} \frac{ \int _{I } e ^ { \kappa x}  | \psi - \kappa|     |v'| ^ { p-1}  |v| }{ \int _{I } e ^ { \kappa x} |v| ^ { p}} 
& \geq 
\frac{e ^ {- | \kappa| }  \min \{ u ^ +( 0) , u ^ - ( 1)   \} ^ p  }{e ^ { |\kappa| } \|u \| _\infty ^ p } 
\int _{I }  | \psi - \kappa|     |v'| ^ { p-1}  |v|    \\
\noalign{\medskip} 
& \geq 
\frac{e ^ {-2 \kappa^*}  \big ( \frac{1}{  2  \cafour   } \big ) ^ {\frac{p}{p-1}}   }{   \cafour  ^ p } 
\int _{I }  | \psi - \kappa|     |v'| ^ { p-1}  |v| \\ \noalign{\medskip} 
 & = 
 \frac{1}{2^ {\frac{p}{p-1}}  }  \frac{e ^ {-2 \kappa^*}    }{  \cafour  ^  {\frac{p^2}{p-1}} } 
\int _{I }  | \psi - \kappa|     |v'| ^ { p-1}  |v|
\end{aligned} 
$$

Next we observe that, since 
$\log f  = \log h + \log \phi$, we have   
$$ ( \log f ) ''   = (\log h  ) '' + (\log \phi ) ''  = \frac{h''}{h} - \Big (\frac{h'}{h} \Big ) ^ 2 - ( \log \phi) '' \leq - \Big (\frac{h'}{h} \Big ) ^ 2  \,, $$ 
where the last inequality holds since $h$ is (positive) concave and $\phi$  log-concave. 

Then, for every  $x\in (0, 1 - z)$, we have, for some $\xi _ x \in (z, z+x)$, 

$$\begin{aligned} | \psi(x) - \kappa|     & = - (\psi (x) - \psi (z) )  =  - \big (  (\log f )' (x+z)  - ( \log f) ' ( z) \big )    \\  \noalign{\medskip} 
& =  - x (\log f )  '' ( \xi _ x)  \geq  x  \Big (\frac{h'}{h} \Big ) ^2   (\xi _x) \,.  
\end{aligned} 
$$ 

Similarly, for every $x \in (1 -z, 1)$ we have, for some $\xi _x \in  ( x - 1 + z, z)$,

$$\begin{aligned} | \psi(x) - \kappa|     & =  \psi (x) - \psi (z)   =    (\log f )' (x- 1+z)  - ( \log f) ' ( z)    \\  \noalign{\medskip} 
& =  ( x  -  1) (\log f )  '' ( \xi _ x)  \geq  (1- x)  \Big (\frac{h'}{h} \Big ) ^2   (\xi _x)   \,.
\end{aligned} 
$$  

Combining the two above estimates, we obtain 
$$ \begin{aligned}
&  \int _{I } | \psi - \kappa|     |v'| ^ { p-1}  |v| 
 \geq  \int _{I }   \big [ x \chi _{(0, 1- z )} + ( 1 - x)    \chi _{(1- z, 1)}  \big ]   \big (\frac{h'}{h} \big ) ^2   (\xi _x)    |v'| ^ { p-1}  |v|  
 \\ 
 \medskip
 &   \ge \frac{1}{\|u\|_\infty^p} \Big [\int _0 ^ z  (z-y)  \big (\frac{h'}{h} \big ) ^2   (\xi _{y-z+1} )      |(u^+) ' (y)| ^ { p-1}  |u^+(y)|  
 \\ \medskip
 &
 +   \int _z ^ 1  (y-z)   \big (\frac{h'}{h} \big ) ^2   (\xi _{y-z} )  |(u^-) ' (y) | ^ { p-1}  |u^- (y) |\Big]\,,
 \end{aligned}
 $$   
  with $\xi _{y-z+1}  \in (y, z) $ and   $\xi _{y-z} \in (z, y)$.

  Recall from \eqref{f:sstima1} that $z$ belongs to the interval $\big [ \frac{1}{2   \cafour  ^ p }, 1 -    \frac{1}{2  \cafour ^ p } \big ]$. Hence,
     if   in the above integrals over $[0, z]$ and over $[z, 1]$ we restrict the integration domain respectively to   $\big [ \frac{1}{4  \cafour ^ p }, z \big ]$   and to 
    $\big [ z, 1 -    \frac{1}{4 \cafour ^ p}   \big ]$, we have 
    $$\xi _{y-z+1}\, ,   \xi _{y-z} \in \big [ \frac{1}{4  \cafour  ^ p }, 1 -    \frac{1}{4  \cafour  ^ p } \big ]\,.$$ 
    We conclude that, letting $\delta :=  \frac{1}{4   \cafour ^ p }$,  
 we have the lower bound 
\begin{equation}\label{f:step3}  
\begin{aligned} \int _{I } | \psi - \kappa|     |v'| ^ { p-1}  |v|  \geq   \frac{1}{\gamma^p} \min _{[\delta, 1- \delta]} \Big ( \frac{h'}{h} \Big ) ^ 2 
    \cdot
     \Big   [  & \int _\delta ^ z  (z-y)    |(u^+) ' | ^ { p-1}  |u^+|  +  \\ 
     \noalign{\medskip} 
       &  \int _z ^  {1 - \delta}   (y-z)     |(u^-) '  | ^ { p-1}  |u^- |  \Big ]  \,
       \end{aligned} 
       \end{equation} 
so we get \eqref{abf.57}.
  
    \bigskip 
{\bf Step 4.} We finally  estimate  the two  integrals appearing at the right hand side of \eqref{f:step3}. By change of variables they can be rewritten as 
  $$  \begin{aligned}   & \int _\delta ^ z  (z-y)     |(u^+) ' | ^ { p-1}  |u^+|   = \int_ 0 ^ {z - \delta } x  |(u^+) ' (z-x) | ^ { p-1}  |u^+|  (z-x)
    \\ 
  \noalign{\medskip}    
&       \int _z ^  {1 - \delta}   (y-z)     |(u^-) '  | ^ { p-1}  |u^- |   = \int _0 ^ { 1- \delta - z} x  |(u^-) ' (z+x) | ^ { p-1}  |u^-|  (z+x) 
      \end{aligned} $$ 
    Thus we are in a position to apply Lemma \ref{l:2cases}, in the former case with $a = a_1:= z- \delta$ and $g =g_1(x):=  u ^+ (z-x)$, in the latter case with $a = a_2:=1- \delta -z$ and $g =g_2(x):=  u ^-(z+x)$.  The inequality given by Lemma \ref{l:2cases} holds true respectively with     
    $$M _1 := M  ( p ,a_1,  \| g _ 1 \| _\infty, \| g_ 1 ' \| _p) \qquad \text{ and } \qquad 
    M _2 := M  (p , a_2, \| g _ 2 \| _\infty, \| g_ 2 ' \| _p) \,$$
    being $M (p ,a, \| g   \| _\infty, \| g  ' \| _p )$ given by \eqref{f:M}
    (for  corresponding constants $( b _0 )_1$ and $(b _0) _2$ given by \eqref{f:b0}). 
    
We need to bound from below $M _ 1$ and $M _2$.     
Looking at \eqref{f:M}, this amounts to give a lower bound for $\| g _i \| _\infty$ and for $(b _0 ) _i$, and an upper bound for $\| g ' _ i \| _p $, for $i = 1, 2$. 
    
    We observe that, from \eqref{f:sstima1}  and the choice of $\delta$, we have 
    $$\min \big \{a_1, a_2 \big \}\geq \frac{1}{4 \cafour ^ p }\,.$$ 
    Therefore, the corresponding constants $( b _0 )_1$ and $(b _0) _2$ given by \eqref{f:b0}  satisfy the lower bounds
$$\begin{aligned}
& (b _0)_1:= \min \Big \{ \frac{a_1}{2},  \Big ( \frac{1}{2} \frac {p-1}{p} \frac{\| g_1 \| _\infty} { \| g_1' \| _p}  \Big )  ^ {\frac{p}{p-1}  }  \Big \}
\geq 
\min \Big \{\frac{1}{8 \cafour  ^ p } ,  \Big ( \frac{1}{2} \frac {p-1}{p} \frac{\| g_1 \| _\infty} { \| g_1' \| _p}  \Big )  ^ {\frac{p}{p-1}  }  \Big \}
 \,,
\\ \noalign{\medskip} 
& (b _0)_2:= \min \Big \{ \frac{a_2}{2},  \Big ( \frac{1}{2} \frac {p-1}{p} \frac{\| g_2 \| _\infty} { \| g_2' \| _p}  \Big )  ^ {\frac{p}{p-1}  }  \Big \}
\geq 
\min \Big \{\frac{1}{8 \cafour ^ p } ,  \Big ( \frac{1}{2} \frac {p-1}{p} \frac{\| g_2 \| _\infty} { \| g_2' \| _p}  \Big )  ^ {\frac{p}{p-1}  }  \Big \} \,.
\end{aligned} 
 $$ 
By the estimate \eqref{f:sstima2} in Step 2 applied at $x  = \delta$ and the choice of $\delta$, we have 
$$|u ^ +(\delta) |^ { p-1}  \geq \frac{1}{2 \|u \| _\infty} - \delta \|u \| _\infty ^ { p-1} \geq \frac{1}{ 4 \| u \| _\infty } \,.$$ 
Therefore,  it holds
$$\| g _ 1\| _\infty \geq \frac{1}{(4\cafour) ^ {\frac{1}{p-1}} }\,.$$ 
On the other hand, we have  
$$ \begin{aligned}
\| g ' _ 1 \| _p & = \int_0 ^ { z - \delta}  |u' ( z-x) | ^ p \, dx = \int _ \delta ^ z | u' (y) | ^ p \, dy  \\ \noalign{\medskip}
\smallskip
& \leq \frac{1}{\delta ^ { m+1} } \int _ \delta ^ z  | u' (y) | ^ p  f ( y) \, dy   \leq  \frac{(\pi _p ) ^ p + 1}{\delta ^ { m+1}  }   = ( 4 \cafour  ^ p) ^ { m+1}   \big ( (\pi _p ) ^ p + 1 \big ) \,.
\end{aligned} 
$$   
Since the analogous estimates hold also for $g_2$, we finally have 
$$M _ i  \geq    M  \Big ( p,  \frac{1}{4 \cafour ^ p } , \frac{1}{(4\cafour ) ^ {\frac{1}{p-1}} }\,  ,
 ( 4 \cafour  ^ p) ^ { m+1}   \big ( (\pi _p ) ^ p + 1 \big ) \Big )  \,.$$ 
 Then, recalling \eqref{f:inspection}, \eqref{f:step3}, and \eqref{f:sstima5},  we conclude that the inequality in the statement of the proposition is satisfied with 
$$
\casix:=\frac{1}{4  \cafour ^ p } \qquad \text{ and } \qquad \cafive:= 2   
  \frac{1}{2^ {\frac{p}{p-1}}  }  \frac{e ^ { - 2 \kappa^* }     }{  \cafour  ^  {\frac{2p^2-p}{p-1} }}    M \Big ( p,   \frac{1}{4 \cafour ^ p } ,   \frac{1}{(4\cafour ) 
^ {\frac{1}{p-1}} }\,  ,
 4 \cafour  ^ p) ^ { m+1}   \big ( (\pi _p ) ^ p + 1 \big ) \Big ) \,,$$  
being $M (\cdot, \cdot, \cdot, \cdot)$ defined by \eqref{f:M}, thus leading to \eqref{f:cafive}.
 \qed

 \bigskip\bigskip\bigskip

  \subsection{Modified Payne-Weinberger partitions and rigidity}\label{s:s2} The following definitions are adapted from \cite{ABF24}:  

\begin{definition}
Given an open bounded convex set $\om \subset \R ^N$, a positive weight $\phi \in L ^ 1 (\om)$, and a function $u\in L ^ p  (\om, \phi \, dx)$ satisfying $\int _\om |u| ^ { p-2} u  \phi  = 0$, 
we call:

\medskip
\begin{itemize}

\item[--]  a {\it $\phi$-weighted measure equipartition  of $u$  in $\om$} a family $\mathcal P _n = \{ \om _ 1, \dots, \om _n \}$, where $\om _i$ are mutually disjoint convex sets  such that $\om = \om _1 \cup \dots \cup  \om _n $ and 
$$\int _{\om _i } |u| ^ { p-2} u  \phi  = 0 \qquad \text{ and } \qquad |\om _i|    = \frac{1}{n} |\om|  \qquad \forall i = 1, \dots, n\,. $$

\medskip

\item[--]  a {\it $\phi$-weighted $L^p$ equipartition of $u$  in $\om$} a family $\mathcal P _n = \{ \om _ 1, \dots, \om _n \}$, where $\om _i$ are mutually disjoint convex sets  such that $\om = \om _1 \cup \dots \cup  \om _n $ and 
$$\int _{\om _i }|u| ^ { p-2} u  \phi= 0 \qquad \text{ and } \qquad \int _{\om _i }   | u  |   ^ p  
\phi   = \frac{1}{n} \int _{\om  } |  u |   ^ p  \phi  \qquad \forall i = 1, \dots, n\,. $$

\end{itemize}

\end{definition}

\bigskip
We point out that, as a straightforward variant of Lemma 28 in \cite{ABF24},  if $\overline u$ is  an eigenfunction for $\mu _p (\Om,\phi )$,
$\phi$-weighted $L^p$ equipartitions of $\overline u$ in $\Om$ enjoy the following  mean value property: 
 \begin{equation}\label{f:meanvalue}
 \mu _ p (\Om, \phi ) \geq  \frac{1}{n}  \sum _{i = 1} ^n  \mu _ p (\Om _ i, \phi) \,.
 \end{equation}

\begin{proposition}[rigidity]\label{ambu20}
For any open, bounded,  convex set $\Omega \subset \R^N$, and any positive weight
 $\phi : \Om \to \R_+$ which is $(\frac 1m)$-concave for some  $m \in \N \setminus \{ 0 \}$, it holds
 $$\mu_p(\Om, \phi) > \Big (  \frac{\pi_p}{D_\Om} \Big ) ^ p \,.  $$
\end{proposition}

\proof 
For every $\vps >0$, let
 $$
\widetilde{\Omega }_\varepsilon := \Big \{ (x,y) \in \R^{N+m} : \, x \in \Omega ,\,  \norma{y}_{\R^m} <\varepsilon \phi ^{ \frac {1}{m}}(x)\Big \}\,.$$
By Corollary \ref{c:collapse} and Payne-Weinberger inequality for the $p$-Laplacian, 
we have 
$$
\mu_p(\Omega,\phi) = \lim _{\e \to 0 }  \mu_p (\widetilde{\Omega}_\varepsilon) \geq \Big ( \frac{\pi_p}{D_\Om}\Big ) ^ p.
$$

In order to show that the  inequality is strict, we 
argue by contradiction. We assume that  
$ \mu _ p (\Om, \phi ) =  \Big ( \frac{\pi_p}{D_\Om}\Big ) ^ p$, we denote by 
$\overline u$ an eigenfunction for $\mu _p (\Om, \phi)$, and   we 
proceed in two steps.

\smallskip
\noindent{\bf Step 1.} Assume $N=2$, and let 
$\mathcal P _n = \{ \Om _1 , \dots,   \Om _n\}$ be a $\phi$-weighted $L ^p$ equipartition of $\overline u$ in $\Om$.
By the mean value property \eqref{f:meanvalue} we have 
$$\begin{aligned}
\Big ( \frac{\pi_p}{D_\Om}\Big ) ^ p   = \mu _ p (\Om, \phi) \geq \frac{1}{n} \sum _{i = 1} ^n \mu _ p (\Om_i, \phi ) & \geq 
\frac{1}{n} \sum _{i = 1} ^n \Big ( \frac{\pi_p}{D_{\Om_i}}\Big ) ^ p  \\ 
\noalign{\smallskip} 
& \geq
\Big ( \frac{\pi_p}{D_\Om}\Big ) ^ p   + p\pi_p^p   \sum _{i = 1} ^n
\frac{(D _\Om -D_{\Om _i}) }{D  _\Om ^ {p+1}}  \,. 
\end{aligned} $$
by the mean value property \eqref{f:meanvalue} and convexity of $x \to x^{-p}$.

We infer that, for every $i = 1, \dots, n$, it holds  
\begin{equation}\label{f:uguaglianze} D  _\Om = D _{\Om _i} \qquad \text{ and } \qquad  \mu _ p (\Om _i, \phi) = \Big ( \frac{\pi_p}{D_\Om}\Big ) ^ p \,.
\end{equation} 
As in the proof of Theorem 31 in \cite{ABF24}, this implies that $\Om$ is a circular sector, and $\Om _i$ are circular subsectors,  the supremum of whose opening angles is infinitesimal as $n \to + \infty$.
We select a decreasing sequence of such sectors, with opening angles tending to zero, and we pass to the limit. We find, for a  line segment $S$ of length  $D_\Om$, 
$$\mu _p (S, x \phi) = \Big ( \frac{\pi_p}{D_\Om}\Big ) ^ p  \,,$$ 
in contradiction with Proposition \ref{p:1d}  (see also Remark \ref{r:pip}). 

\smallskip
\noindent{\bf Step 2.}   Assume $N \geq 3$, and let  $\mathcal P _n = \{ \Om _1 , \dots,   \Om _n\}$  be
 a $\phi$-weighted $L ^p$ equipartition of $\overline u$ in $\Om$, 
whose cells are arbitrarily narrow in $(N-2)$ orthogonal directions.  In the limit as $n \to + \infty$, any sequence of cells  $\Om _{k (n) }  \in \mathcal P _n$ 
 converges, up to a subsequence, to a degenerate convex set,  which may be 
either one-dimensional or two-dimensional. Actually, since the equalities
\eqref{f:uguaglianze} are still satisfied, 
 any  limit set has diameter $D_\Om$,  and this readily leads to exclude by a geometric contradiction argument 
 the case when all the limit sets are one-dimensional. Hence,  we can select a sequence $\Om _{k (n )} \in \mathcal P _n$ converging to  a two-dimensional convex set $\Om _0$ which satisfies, for a suitable $(\frac{1}{N-2})$-concave function $h$, the equality
$$
\mu _ p (\Om_0,h  \phi  ) =  \Big ( \frac{\pi_p}{D_\Om}\Big ) ^ p  \,, 
$$ in contradiction with Step 1. 
\qed

\bigskip 
\subsection{Area control for the cells}\label{s:s3} In this section we prove that, given a 
$\phi$-weighted $L^p$ equipartition  of  a first eigenfunction for $\mu _ p (\Om, \phi)$, 
we are able to control the Lebesgue measure of the cells. This  will be a crucial point in order to conclude the proof of Theorem \ref{t:Ngap}.   
As a preliminary ingredient, we need to establish a nonlinear Kroger-type inequality.   Results in this spirit have been recently proved in \cite{Henrot-Michetti}.    
\begin{lemma}\label{l:kroger} 
	Let $\Omega$ be an 
	open bounded convex domain $\Omega$ in $\R ^N$ with diameter $D_\Om$  	and 
$\phi : \Om \to (0,+\infty)$ be a  $(\frac {1}{m})$-concave function  on $\Om$. Then there holds 
		\begin{equation}
		\mu_p(\Omega, \phi) \leq \frac{\cakroger(p, N+m)} {D_\Omega^p},
	\end{equation}
where
\begin{equation}\label{f:cakroger}
\cakroger = \cakroger(p, N):=     (p+1)2^{\frac{N+  3 p}{2}}     N ^ { N+p}   \,.  
\end{equation} 
\end{lemma}
\begin{proof}
By a collapsing argument, it is enough to prove that, for $\Om$ as in the statement, it holds 
	$$
		\mu_p(\Omega) \leq \frac{\cakroger(p, N)} {D_\Omega^p}.
$$
Moreover, up to a translation, we can assume that the John ellipsoid $\mathcal E$ of $\Om$ is centred at the origin.  
Writing $x \in \R ^N$ as $x= (x_1, x' ) \in \R \times \R ^ {N-1}$,   choosing $c$ such that $\int_\Omega  |x_1-c | ^ { p-2} (x_1-c) =0$, as 
and taking $u ( x) = x_1-c$ 
a test function in 
the variational characterization of $\mu_p(\Omega)$,  we have 
			$$ \mu_p(\Omega)=\min_{\substack{u\in W^{1, p}(\Om)\sm \{0\}\\ \int_\Omega  |u | ^ { p-2} u =0}} \frac{\displaystyle{\int_\Omega \abs{\nabla u}^p}}{\displaystyle{\int_\Omega |u |^p}} \leq\frac{\displaystyle{\int_\Omega \abs{\nabla x_1}^p}}{\displaystyle{\int_\Omega |x_1-c |^p}}= \frac{\abs{\Omega}}{\displaystyle{\int_\Omega |x_1-c |^p}}\,. $$

Hence, 
$$
	\mu_p(\Omega)\leq  \frac{|\Om|} { \min_{c \in \R}  \int_\Omega |x_1-c |^p  }\,.
 $$

Since $\mathcal E$ is contained into $\Om$ and it is centred at the origin, we have   
$$
	\label{elipsoid_int}
	\min_{c \in \R} 	\int_\Omega \abs{x_1-c}^p \geq \min_{c \in \R} \int_{\mathcal E}  \abs{x_1-c}^p  = \int_{\mathcal E}  \abs{x_1}^p \,.
$$

Some explicit computations give 
$$
\begin{aligned}
		 \int_{\mathcal E}  \abs{x_1}^p  \, dx &   = \int_{-a_1}^{a_1} \abs{x_1}^p   \mathcal{H}^{N-1} \big ( \{ x'\ :  (x_1 , x') \in \mathcal E \} \big ) \, dx_1    \\
		 &=     \int_{-a_1}^{a_1} \abs{x_1}^p   \om_{N-1}   a_2 \dots a_N \left(1- \frac{x_1^2}{a_1^2}\right)^{\frac{N-1}{2}}\\
		 &\geq   \int_{-\frac{a_1}{\sqrt{2}}}^{\frac{a_1}{\sqrt{2}}} \abs{x_1}^p     \om_{N-1}    a_2 \dots a_N \left(1- \frac{x_1^2}{a_1^2}\right)^{\frac{N-1}{2}}\\
			 &\geq  \frac{1}{2^{\frac{N-1}{2}}}   \om_{N-1}    a_2 \dots a_N \frac{2}{p+1} \left(\frac{a_1}{\sqrt{2}}\right)^{p+1}\\
		 &=\frac{1}{(p+1)2^{\frac{N+p-2}{2}}}    \om_{N-1}     a_2 \dots a_N  (a_1)^{p+1}\\
		 &=    \frac{1}{(p+1)2^{\frac{N+p-2}{2}}}    \frac{ \om_{N-1}   }{ \om_{N}    }       \abs{\mathcal E } (a_1)^{p}\\
		 & \geq    \frac{1}{(p+1)2^{\frac{N+   3  p-2}{2}}}   \frac{ \om_{N-1}   }{ \om_{N}    }   \frac{\abs{\Omega}}{N^N} \frac{D_\Omega^p}{ N^p}.
\end{aligned}
$$ 
We conclude that
$$
	\mu_p(\Omega) \leq     
	(p+1)2^{\frac{N+  3 p-2}{2}}  \frac{ \om_{N}   }{ \om_{N-1}    }    N ^ { N+p}  \frac{1}{D _\Om ^p} 
	\leq      
	(p+1)2^{\frac{N+  3 p}{2}}     N ^ { N+p}  \frac{1}{D _\Om ^p}. 
	$$
\end{proof}

\begin{proposition}\label{p:proparea}    Let  
\begin{equation}\label{f:caseven} 
 \caseven = \caseven (p, N+m) :=      \frac{ 1}{2^{N+m}}   \frac{1}{  \caone  ^ p   \cdot \ka _{kroger}  ^ {N+m} }  
\end{equation} 
  where $\caone$ is the constant 
defined by \eqref{f:caone} computed at $N+m$,  and  $\ka _{kroger}$ is the constant defined by \eqref{f:cakroger} computed at $(p, N+m)$.   

Then, if $\Om \subset \R^N$ is an open bounded convex set, 
$\phi$ is a positive $\big ( \frac 1m\big )$-concave weight in $L ^ 1 (\Om)$ and  $\overline u$ is a first eigenfunction for $\mu _ p (\Om, \phi)$ normalized in $L ^ p (\Om, \phi)$,  
any  cell $\om$ of a $\phi$-weighted $L^p$ equipartition of $\overline u$ in $\Om$ composed by $n$ cells satisfies
	\begin{equation*}
		|\omega| \geq \caseven \frac{|\Omega|}{n}.\;\;\;
	\end{equation*}
\end{proposition}

\proof  By Corollary \ref{c:collapse}, we have
	\begin{equation*}
		 \frac{1}{n\int _{\omega}\phi }= \frac{\int_{\omega} \phi \overline u^p }{\int _{\omega} \phi } \leq \norma{\overline u}^p_{L^\infty (\Om)} \leq     \caone ^p     
		 \mu_{  p  } (\Omega,\phi )^{N+m} \frac{D_\Om ^{p(N+m)}}{  \int_\Om \phi}.
	\end{equation*}

By Lemma \ref{l:kroger},  we have that
$$\mu_{  p  } (\Om,\phi ) ^ { N+m} D_\Om ^{p(N+m)} \leq \ka _{kroger} ^  {N+m} \,, $$

which leads to 
$$\int _{\omega}\phi \geq   \frac{1}{  \caone  ^ p   \cdot \ka _{kroger}  ^ {N+m} }  \frac{1}{n} \int _\Om \phi\,.$$

Now we observe that, since $\phi$ is $\big ( \frac 1m)$-concave, it belongs to 
$L^\infty(\Om)$ and, if  we assume with no loss of generality that $0 \in \overline \Om$ is a maximum point for $\phi$, it holds 
$\phi (x) \ge  \frac {1}{2^m} \|\phi\|_{L^\infty(\Om)}$ for every  $x \in \frac 12 \Om $. We infer that 
	
	$$\abs{\omega} \|\phi \|_{L^\infty(\Om)} \ge \int _{\omega} \phi  \ge  
	 \frac{1}{  \caone  ^ p   \cdot \ka _{kroger}  ^ {N+m} } \frac{1}{n}  \int _{\Omega} \phi  \ge    \frac{1}{  \caone  ^ p   \cdot \ka _{kroger}  ^ {N+m} }  \frac{1}{n}  \frac{ 1}{2^{N+m}}  \|\phi \|_{L^\infty(\Om)} \abs{\Om}\,.$$

 \qed

\bigskip 

\subsection{Proof of Theorem \ref{t:Ngap} in dimension $N = 2$}\label{s:s4}     It is not restrictive to prove the statement for $D _\Om = 1$. 
Moreover, we can  fix a system of coordinates such that a diameter is the horizontal segment $[0, 1] \times \{ 0 \}$.   
Since we are in dimension $N=2$,  $a_2$ is comparable to the width $w_\Om$ of $\Om$, or equivalently to its depth $\eta_\Om$ in vertical direction (namely the maximum of the lengths of vertical sections of $\Om$).

We split the proof of Theorem \ref{t:Ngap} in two steps: first, we prove the quantitative inequality for narrow sets, and then 
for arbitrary sets. 
   This is done respectively in the two lemmas hereafter, which  are proved separately below.

	\begin{lemma} \label{l:lemma1} 	  Let $N = 2$, and  let $\Om$  and $\phi$ as in the assumptions of Theorem \ref{t:Ngap}. Assume in addition that $D_\Om =1$ and that $\eta _\Om \leq \eta _0$, with 
 
	\begin{equation}\label{f:eta0} 
	 \eta _0 =  \frac{1}{2} \sqrt{ \frac{2 \cafive}{3} }    \,.
	 \end{equation}    
		 Then it holds
\begin{equation}
		\label{f:formula11}
		\mu_p (\Omega, \phi)\ge  ({\pi_p}) ^ p   +      \frac{\caseven ^2 }{ (8 \cdot  256) ^ 2}      \eta _\Om ^2 \,.
	\end{equation}

		   \end{lemma}    
		   
		  \begin{lemma} \label{l:lemma2}  
		    Let $N = 2$, and  let $\Om$  and $\phi$ as in the assumptions of Theorem \ref{t:Ngap}. Assume in addition that $D_\Om =1$. Then it holds 
\begin{equation}	\label{f:formula22}
\mu_p (\Omega, \phi)\ge (\pi _p ) ^ p + 
    \frac{ (\caseven) ^ 2  \eta _0 ^ 2 }{(7 \cdot 16 \cdot 256  )^2}   
   \eta _\Om ^ 2    \,,
	\end{equation}  	  
	where  $\eta_0$ is given by \eqref{f:eta0}. 
		\end{lemma}

		  {\bf Proof of Lemma \ref{l:lemma1}}. 
		  By an approximation argument,  it
 is not restrictive to assume that
 $\Om$  and $\phi$ are smooth,  so that a first eigenfunction $\overline u$ for $\mu _ p (\Om, \phi)$ belongs to $\mathcal C ^ 2 (\overline \Om)$.  

		  For every $n\in \N$, we consider a
 $\phi$-weighted $L^p$ equipartition $\{ \Om _1, \dots, \Om _n \}$ of $\overline u$ in $\Om$. 
 The leading  idea is that, 
 if for a good proportion of cells (say $\frac{n}{2}$), we have a lower bound with an excess term, of the form $\mu _ p (\Om _i , \phi)  \geq (\pi _ p ) ^ p + \delta$ for some $\delta >0$, 
then by using the mean value property \eqref{f:meanvalue} we obtain 
\begin{equation}\label{f:pezzetto} 
 \mu _ p (\Om, \phi)  \geq \frac{1}{n} \sum _{i = 1} ^n \mu _ p (\Om_i, \phi  ) \geq \frac{1}{n} \Big [ \frac{n}{2} {(\pi _ p ) ^p  } + \frac{n}{2}  { \big  ((\pi _ p ) ^p  + \delta \big ) } \Big ] = (\pi _ p ) ^p   + \frac{\delta}{2} \,.
 \end{equation} 
Thus we distinguish a list of cases, trying to individuate 
which sub-families of cells  may give rise to an excess, and  taking care to control their proportion. 

Relying on the results  proved in the previous sections, 
we are in a position to closely follow \cite{ABF24}, 
   with the additional purpose of tracking all the involved constants. 
Below  $\  c_0 \  $ denotes  a number in $(0, 1)$, whose value will be fixed at the end of the proof.  
 In each case we determine the validity of  inequality \eqref{f:formula11} for some constant  depending on  $p$, $m$ and also on $ c_0  $ 
 (but independent of $\Om$). The final constant
will be given by the minimum among such constants found in the different cases, once fixed $  c_0  $.  
Throughout the proof we write for brevity $\eta$ in place of $\eta _\Om$. 

In order to track the geometry of the cells of the partition, we are going to fix the diameter of $\Om$ as being the segment $ [- \frac 1 2, \frac 1 2 ] \times \{0\} $.

\medskip
-- {\it Case 1}: 
For a sequence of integers $n$,  there exists $\mathcal I _n \subset  \{1, \dots, n\}$  with  ${\rm card} ( \mathcal I _n ) \geq \frac{n}{2}$, such that, for every $ i \in \mathcal I _n$,  the diameter $D_i$ of $\Om _i$ 
satisfies $D_i \leq \sqrt{ 1 -   c_0    \eta ^ 2}$. Then, for cells $\Om _i$ with
$i \in \mathcal I _n$, we have
 
$$\mu _ p (\Om_i, \phi )  \geq 
 \Big ( \frac{\pi_p}{D_i}\Big ) ^ p  \geq 
 \frac{ (\pi_p )^p  }{ ( 1 -  c_0  \eta ^ 2 ) ^  {\frac{p}{2} }   }
  \geq  
(\pi _ p) ^ p \big  ( 1 +  \frac{p}{2}   c_0  \eta ^ 2 \big  )  \,,     $$ 
where the last inequality holds by convexity on $(0, 1)$ of the map $t \mapsto (1-t) ^ { - \frac{p}{2}}$. 
In view of \eqref{f:pezzetto},  we conclude that   in Case 1 
the quantitative inequality holds under the form 
   \begin{equation}\label{f:caeight}
\mu_p (\Omega, \phi)\ge   (\pi_p  ) ^ p   +     (\pi _ p)  ^ p \frac{p}{4}  c_0    \eta _\Om ^2.  
 \end{equation}

\smallskip

-- {\it Case 2}: For a sequence of integers $n$, 
there exists $\mathcal I' _n \subset  \{1, \dots, n\}$  with  ${\rm card} ( \mathcal I' _n ) \geq \frac{n}{2}$ such that,  for every $i \in \mathcal I' _n$,  
the diameter $D_i$ of $\Om _i$ satisfies  $D_i \geq \sqrt{ 1 - c_0 \eta ^ 2}$. In this case, choosing  
 
\begin{equation}\label{f:choicea1} 
\eta _0 <  \frac 16  \, , \qquad  \text {and } \qquad 6 \eta _0 ^ 2  \leq a \leq  \frac{1}{4},
\end{equation}   the cells $\Om _i$ for $i \in \mathcal I ' _n$  must intersect the vertical lines 
$ \{x_1= - \frac{1-a} {2} \}$ and $\{x_1=\frac{1-a} {2} \ \}$.   Indeed, if  $\Om _i$ would not intersect such vertical lines, $\Om$ would have a projection  onto the vertical axis of length  strictly larger that 
$$\sqrt {D_i ^ 2 - ( 1 - \frac{a}{2} \big ) ^ 2 } \geq \sqrt { 1-   c_0   \eta ^ 2 -  ( 1 - \frac{a}{2} \big ) ^ 2  } >  \eta _0 \sqrt{ 5 -  c_0 }   \geq 2 \eta _0 \geq w ^ \perp\,,$$ 
yielding a contradiction.   We infer that,
for $i \in \mathcal I' _n$, 
 the cells $\Om _i$ 
can be ordered vertically upon the fixed diameter $[- \frac 1 2, \frac 1 2] \times \{ 0 \}$ of $\Om$,  and their intersections with the strip $\big ( - \frac {1-a} 2  , \frac {1-a} 2 \big ) \times \R$  have a polygonal boundary.  

\smallskip
 
-- {\it Case 2.1}: For a sequence of integers $n$,  there exists
$ \mathcal J _n \subset\mathcal  I' _n$, with  ${\rm card} ( \mathcal J_n ) \geq \frac{n}{4}$, such that,  for every $i \in \mathcal J _n$,  the polygonal boundary of 
$\Om _i$  inside the strip $\big ( - \frac {1-2 a} 2  , \frac {1-2 a} 2 \big ) \times \R$ contains some vertex.
 Such a vertex of $\Om _i$ inside the strip belongs also to the boundary of a neighbouring cell $\Om _j$, whose  the diameter $D_j$  satisfies
$$D_j \leq \sqrt{ ( 1-a) ^ 2 + 4 \eta ^ 2 } \leq \sqrt { 1 - 6 \eta _0 ^ 2 }  \,,
$$
where the last inequality holds by the choice of $a$.  Since $\Om _j$ 
can touch at most another cell $\widetilde \Om _i$ with $i \in \mathcal J _n$, we conclude that there exists $\mathcal W _n \subset \{ 1, \dots, n \}$ with ${\rm card} ( \mathcal W _n) \geq \frac{n}{8}$,   such that  
$$\mu _ p (\Om_j, \phi )  \geq 
 \Big ( \frac{\pi_p}{D_j}\Big ) ^ p  \geq 
 \frac{ (\pi_p )^p  }{ ( 1 - 6 \eta_0 ^ 2 ) ^  {\frac{p}{2} }   }
  \geq  
(\pi _ p) ^ p \big  ( 1 +   3 p \eta_0 ^ 2 \big  )  \qquad \forall j \in \mathcal W_n\,. $$ 
Applying  \eqref{f:pezzetto}, we conclude that,  in Case 2.1, the  quantitative inequality 
holds under the form 
 
 \begin{equation}\label{f:canine}
 \mu_p (\Omega, \phi)\ge   (\pi_p ) ^ p   +    (\pi _ p)  ^ p   \frac{3 p}{8}    \eta_\Om^2. 
\end{equation}

\smallskip

-- {\it Case 2.2}: For a sequence of integers $n$, there exists
$ \mathcal  J' _n \subset \mathcal  I'_n$, with ${\rm card} ( \mathcal  J' _n ) \geq \frac{n}{4}$, such that, for every $i \in \mathcal J' _n$,  
$\Om _i $  has no vertex in the strip.
  For such cells,   when $\eta$ is sufficiently small, the 
profile function $h _i$  in direction orthogonal to their diameter is affine on $\big (- \frac{ D_i-4a}{2}, \frac{D_i-4a}{2} \big ) \times \{0\}$
(using a local coordinate system in which the diameter of $\Om _i$ is along the horizontal axis and centred at the origin). 
Specifically, 
a geometric argument
shows that this is true provided 
		\begin{equation}\label{f:arcsin}
			\frac{a}{\cos (\arcsin   (\frac {2\eta}{\sqrt{1-  c_0  \eta^2}}))}+ \eta   \frac {2\eta}{\sqrt{1-   c_0  \eta^2}}\le 2 a 
	\end{equation} 
(the argument is based on the fact that the angle $\alpha _i$ formed by the horizontal diameter of $\Om$ and the diameter of $\Om _i$ does not exceed $\arcsin 
\Big ( \frac{ 2 \eta} {\sqrt { 1 -   c_0   \eta ^ 2}} \Big )$, see the proof of Theorem 1 in \cite{ABF24} for the detailed explanation).  
We claim that condition \eqref{f:arcsin} is satisfied by choosing 
\begin{equation}\label{f:eta01}  \eta_0 \le \frac 16 \,. 
\end{equation}
  which is always true by \eqref{f:choicea1}.  

Indeed, if \eqref{f:eta01} holds,  we have 
$$\frac {2\eta}{\sqrt{1-   c_0  \eta^2}}  \le \frac {2\eta_0}{\sqrt{1- \eta_0^2}} \le 4 \eta_0 $$
and
$$\cos (\arcsin   (\frac {2\eta}{\sqrt{1-   c_0  \eta^2}})) \ge \cos (\arcsin   (4 \eta_0))\ge \cos  \big (\arcsin   (\frac 23) \big ) \ge \frac 23\,,$$
so that the inequality \eqref{f:arcsin} is satisfied as soon as  $$\frac 32 a+ \frac {2\eta_0^2}{\frac  {  \sqrt {35}  } {6}}\le 2a,$$
which is true from the choice $a \ge 6 \eta_0^2$ made above (cf. \eqref{f:choicea1}). 
  
 Then we distinguish two further sub-cases.

\smallskip

-- {\it Case 2.2.1}: For 
 a sequence of integers $n$,  there exists 
$  \mathcal  S _n\subset {\mathcal J}'_n$,  with  ${\rm card} ( \mathcal  S_n ) \geq \frac{n}{8}$, such that, 
   $$\displaystyle \frac{ \int _{\Om _i}  |\nabla \overline u| ^ 2    \phi  }{\int _ {\Om _i}  | \overline u| ^ 2    \phi  }  
\geq  (\pi _ p) ^ p +  \frac{1}{4}   \qquad \forall i \in \mathcal S _n \,.$$  
  In this case, applying    \eqref{f:pezzetto}, we conclude that,  
  in Case 2.2.1, the inequality \eqref{quantitative3dp} holds under the form
 
  \begin{equation}\label{f:caten}
\mu_p (\Omega, \phi)\ge   (\pi_p ) ^ p   +    \frac{1}{32}    \eta_\Om^2. 
\end{equation}

\smallskip

-- {\it Case 2.2.2}: 
For 
 a sequence of integers $n$, there exists 
$  \mathcal  S _n' \subset \mathcal  J'_n$  with ${\rm card} (\mathcal   S'_n ) \geq \frac{n}{8 }$, such that,  for every $i \in \mathcal S'_n$, 
  $$\displaystyle \frac{ \int _{\Om _i}  |\nabla \overline u| ^ 2    \phi  }{\int _ {\Om _i}  | \overline u| ^ 2    \phi  }  
\leq  (\pi _ p) ^ p +  \frac{1}{4}   \qquad \forall i \in \mathcal S' _n \,.$$

  We observe that  there exists $\mathcal   S _n'' \subset \mathcal  S'_n$,  with  ${\rm card} ( \mathcal  S'' _n ) \geq \frac{n}{16}$, such that  
$$|\Om_i| \leq \frac{17}{n} |\Om| \quad \forall i \in\mathcal   S _n''\,.$$   

 For every $i \in \mathcal S _n ''$, 
 by applying Lemma 26 in \cite{ABF24} we obtain that, denoting by 
  $I _{ D_i}$ a diameter of $\Om _i$ and by $h _ i$ the function   giving the $\mathcal H ^ {1}$  measure of the section of $\Om _i $ in direction orthogonal to $I _ { D_i}$,   for $n$ sufficiently large it holds 
    \begin{equation}\label{f:formula}
 \frac{ \int _{\Om _i}  |\nabla \overline u| ^ 2    \phi }{\int _ {\Om _i}  | \overline u| ^ 2    \phi    }    
\geq   \Big \{ \mu _ p ( I _ {D_i}, h_i  \phi ) - \frac{  \alpha _n  |\Om_i|  }{\frac{1}{n}   } \Big  [ 1 +  \frac{3}{2} \mu _ 1 ( I _ {D_i  } , h_i  \phi )  \Big ] \Big \} \,, 
 \end{equation}   
  where $\alpha_n$ is an infinitesimal as $n \to+ \infty$ (independent of $i$). 
 From the condition $|\Om_i| \leq \frac{17}{n} |\Om|$,   the above inequality implies
$$ \displaystyle \frac{ \int _{\Om _i}  |\nabla \overline u| ^ 2    \phi  }{\int _ {\Om _i}  | \overline u| ^ 2    \phi   }        \geq   
  \Big \{ \mu _ p ( I _ {D_i}, h_i  \phi )  \Big [   1 -   \frac{3}{2}  \cdot 17     |\Om|  \cdot \alpha_n    \Big ]   - 
    17   |\Om| \cdot   \alpha_n  \Big \} \,.
$$
Since $\alpha_n$ is infinitesimal as $n \to + \infty$,  for $n$ sufficiently large it holds 

$$\frac{ \int _{\Om _i}  |\nabla \overline u| ^ 2    \phi }{\int _ {\Om _i}  | \overline u| ^ 2    \phi    }     \geq \frac{ (\pi _ p) ^ p +  \frac{1}{2}   }{ (\pi _ p) ^ p + 1  } 
\mu _ p (I _ {D_i}, h_i \phi )- \frac{1}{4} \,,$$ 
or
$$ \mu _ p (I _ {D_i}, h_i \phi )  \leq \frac{  \Big [  \frac{ \int _{\Om _i}  |\nabla \overline u| ^ 2    \phi }{\int _ {\Om _i}  | \overline u| ^ 2    \phi    }  + \frac{1}{4}   \Big ]  }{ (\pi _ p) ^ p +  \frac{1}{2} }  ((\pi _ p) ^ p + 1 ) \leq   (\pi _ p) ^ p + 1  \,. $$

We are thus in a position to apply   Proposition \ref{p:1d}, more precisely the consequent  inequality \eqref{f:1dqbis.1}. 
  To that aim, we  need to impose the further condition
\begin{equation}\label{f:choicea3}
 2   a \leq \casix\,.
\end{equation}

Provided $D_i \geq d_p$ (being $d_p$ defined by \eqref{f:dp}), we have  
  $$\mu _ p (I _ {D_i}, h_i \phi) \geq (\pi _ p) ^ p +   \cafive  \frac 23 \min _{[ - \frac{D_i}{2} + \casix, \frac{D_i}{2}  - \casix]} \Big ( \frac{h_i'}{h_i} \Big ) ^ 2   \qquad \forall i \in \mathcal S'' _n\,.$$
  
  Notice that, by the definition of $d_p$, if $D _i \leq d_p$, it holds $\mu _ p \geq (\pi _p) ^ p + 1$, so that we are reduced back to the situation of  Case 2.2.1, and the inequality \eqref{quantitative3dp} is satisfied with 
$\ka _0$ replaced by    $ \frac{1}{32}   $.

Then we continue distinguishing two final sub-cases. 
\smallskip

--  {\it Case 2.2.2.1}: For a sequence of integers $n$, 
there exists  $\mathcal  Z _n \subset \mathcal  S'' _n$,  with  ${\rm card} ( \mathcal  Z_n ) \geq \frac{n}{32}$, such that 
$$\frac 2 3 \cafive  \min _{[ - \frac{D_i}{2} + \casix, \frac{D_i}{2}  - \casix]} \Big ( \frac{h_i'}{h_i} \Big ) ^ 2      \geq  c_0   \eta ^ 2 \qquad \forall i \in \mathcal Z_n\,.$$ 
In this case, by the usual dimension reduction argument (cf. \eqref{f:formula}) and 
and  mean value inequality
(cf. \eqref{f:pezzetto}),   we obtain  that the quantitative inequality holds under the form  

\begin{equation}\label{f:caeleven}
 \mu_p (\Omega, \phi)\ge   (\pi_p ) ^ p   +    \frac{  c_0 }{32}       \eta_\Om^2.  
\end{equation} 

\smallskip
 
--  {\it Case 2.2.2.2}: For a sequence of integers $n$, 
there exists $\mathcal Z' _n \subset\mathcal   S''_n$,  with  ${\rm card} ( \mathcal   Z'_n ) \geq \frac{n}{32}$, such that, for every $ i \in \mathcal Z'_n$,  
$$\frac 2 3 \cafive  \min _{[ - \frac{D_i}{2} + \casix, \frac{D_i}{2}  - \casix]}  \Big ( \frac{h_i'}{h_i} \Big ) ^ 2      \leq  c_0   \eta ^ 2   \qquad \forall i \in \mathcal Z' _n\,.$$   
To conclude the proof, it remains to show that this case cannot occur for a suitable choice of $  c_0  $, $a$ and $\eta_0$. 
Assume that 

\begin{equation}\label{f:eta02}
 \eta_0 \le \frac 1 2 \sqrt { \frac{2 \cafive}  {3} } \,. 
 \end{equation}
  Then, since $  c_0    < 1$,
the above inequality  implies that the minimum and the maximum of $h _i$ on the interval
$[ - \frac{D_i}{2} + \casix, \frac{D_i}{2}  - \casix] $ satisfy 
\begin{equation}\label{f:hminmax}\frac{ h _ i ^ {min}  }{ h _ i  ^ {max} }  \geq \frac{1}{2} \,.
\end{equation} 

  Provided $a$ and $\eta _0$ are suitably chosen,  the  inequality \eqref{f:hminmax} remains valid replacing
$h _i ^ {min}$ and $h _ i ^ {max}$  respectively by the lengths $L _i ^ {min}$ and $L _ i ^ {max}$ 
of the intersections of $\Om _i$ with 
$ \{x_1= - \frac{1-4a} {2} \}$ and $\{x_1=\frac{1-4a} {2} \}$, since it holds
\begin{equation}\label{f:Lminmax}
\frac{ L _ i ^ {min}  }{ L _ i  ^ {max} }  \geq \frac{ h _ i ^ {min}  }{ h _ i  ^ {max} }  \geq \frac{1}{2} \,.
\end{equation} 
Precisely, thanks to a geometric argument based on Thales Theorem (see the proof of Theorem 1 in \cite{ABF24}),
 the condition under which   \eqref{f:Lminmax} is valid that 
\begin{equation}\label{f:thales} a \geq h _ i ^{max} \cdot  \sin \alpha _i 
\end{equation} 
Since 
$$h _ i ^{max} \leq  2 h _ i ^{min}  \leq \frac{4 \eta}{D_i} \leq  \frac{4 \eta}{\sqrt { 1 -   c_0  \eta ^ 2}}   \quad  \text{ and } \quad  \sin \alpha _ i \leq  \frac{2 \eta}{\sqrt { 1 -   c_0   \eta ^ 2}}  \,,$$ 
we see that  \eqref{f:thales} is fulfilled provided
$$\frac{ 8 \eta ^ 2 }{ 1 -   c_0   \eta ^ 2} \leq  a\,.$$
 Since, from \eqref{f:eta01}, we have
$$\frac{ 8 \eta ^ 2 }{ 1 -  c_0   \eta ^ 2} \leq  \frac{ 8 \eta _0^ 2 }{ 1 -  \eta _0^ 2}  \leq \frac{36\cdot 8}{35} \eta _ 0 ^ 2\,,$$ 
 in order  to have \eqref{f:Lminmax}, it is enough to choose $a$ so that

\begin{equation}\label{f:choicea2} 
a \geq \frac{36\cdot 8}{35} \eta _ 0 ^ 2 \,,
\end{equation}

By using \eqref{f:Lminmax} and 
the control on the area of each cell given by Proposition \ref{p:proparea},  we obtain 
 $$L _i ^ {min} \geq \frac{1}{2} L _i ^ {max} \geq \frac{1}{2} |\Om_i| \geq \frac{\caseven}{2n} |\Om| \geq  \frac{\caseven}{4 n} \eta\,.$$

Now we are ready to reach a contradiction, for a suitable choice of the constant $  c_0  $. 
The geometric argument is the same as in   \cite{ABF24}. Specifically, we consider the quadrilateral $A_1B_1A_2B_2$, being $A_1 B_1$  and $A_2B_2$ respectively the diameters of the bottom and top  cells in $\mathcal Z'_n$. From \eqref{f:Lminmax} we have 
$$\big (  \overline{A_ 1 A_2} \wedge \overline{B_1 B_2}  \big ) \geq  \frac {n}{32}  \frac{1}{2} L _i ^ {min} \geq  \frac{\caseven}{256} \eta  \,.$$ 
while recalling that $\mathcal  Z' _n \subset \mathcal I ' _n$, we have 
$$ \big (  \overline{A_ 1 B_1} \wedge \overline{A_2 B_2}  \big ) \geq  \sqrt {1 -   c_0   \eta ^2 }$$ 
Eventually, combining the last two inequalities and applying the generalized  Pythagorean Theorem, we see that one of the two diagonals of the quadrilateral would have length  strictly larger than $1$, provided
 $$c_0<  \Big ( \frac{\caseven}{256}  \Big ) ^ 2 \,.$$  
 Summarizing, we have shown that Case 2.2.2.2. cannot hold, if we choose 
  
\begin{itemize}
\item{} $\eta _0$ satisfying the first inequality in \eqref{f:choicea1}, and \eqref{f:eta02}\,, 
\smallskip

\item{} $a$ satisfying  the upper bound in \eqref{f:choicea1}, \eqref{f:choicea3}, and \eqref{f:choicea2},  
\smallskip

\item{} $  c_0  :=  \min \{ \frac{ (\caseven) ^ 2 }{2 (256)^2}  , 1 \}
= \frac{ (\caseven) ^ 2 }{2 (256)^2}  $ (the latter equality being a consequence of \eqref{f:caseven})\,.
\end{itemize}
\smallskip 

Noticing that \eqref{f:choicea3} combined with    \eqref{f:choicea2}    demands that 
$$\eta_0 \leq \frac{\sqrt  {35 \cdot \casix }}{  24 } \,, $$

our proof is then achieved, by choosing
$$ \eta _0  = \min  \Big \{ \frac{1}{6}, \frac{1}{2} \sqrt{ \frac{2 \cafive}{3} } , \frac{\sqrt  {35 \cdot \casix }}{  24  }     \Big \} \,.  $$   
and taking in front of the flatness term the smallest among the constants appearing in the quantitative inequalities
\eqref{f:caeight}, \eqref{f:canine}, \eqref{f:caten} and \eqref{f:caeleven}, 
 obtained respectively
in Case 1, Case 2.1, Case 2.2.1, and Case 2.2.2.1.
Its value is given by 
$$  \min \big \{(\pi _ p)  ^ p \frac{p}{4}   c_0   , (\pi _ p)  ^ p   \frac{3 p}{8} , \frac{1}{32}  , \frac{  c_0 }{32}  \big \} = \frac{c_0}{32} = 
\frac{ (\caseven) ^ 2 }{(8 \cdot 256)^2}  
\,,   $$
where the first equality is a straightforward consequence of the inequalities $c_0 \leq 1$ and $(\pi _p) ^ p \geq 2$ (cf.\ Remark \ref{r:pip}).

We finish the proof observing that 
$$ \eta _0  = \min  \Big \{ \frac{1}{6}, \frac{1}{2} \sqrt{ \frac{2 \cafive}{3} } , \frac{\sqrt  {35 \cdot \casix }}{  24   }     \Big \} =  \frac{1}{2} \sqrt{ \frac{2 \cafive}{3} }.$$
   In order to prove that  $\frac{1}{2} \sqrt{ \frac{2 \cafive}{3} }\le \frac{1}{6}$ it is enough to have 
$\cafive\le \frac 16$, which is a consequence of $\gamma >1, M<1, m\ge 0$ and $2e^{-2^3}<\frac 16$. 
In order to prove that $\frac{1}{2} \sqrt{ \frac{2 \cafive}{3} }\le \frac{\sqrt  {35 \cdot \casix }}{  24   }  $ it is enough to have $\cafive \le \frac 14 \casix$, which is a consequence of $\gamma >1, M<1, m\ge 0$ and $2e^{-2^3 y^2}< \frac {1}{ 16  y}$ for $y \ge 1$, with $y=\gamma^p$.

\qed

 \bigskip\bigskip 
 {\bf Proof of Lemma \ref{l:lemma2}}.  Let  $\Om$  and  $\phi$ satisfy the assumptions of Theorem \ref{t:Ngap}.  We divide the proof in three steps. 
 
 \medskip
 {\bf Step 1}. It holds
 $$D_\Om \leq 1 \, , \eta _\Om \leq \frac 2 3 \eta _0 \ \Rightarrow \ \mu _p (\Om, \phi) 
 \geq (\pi _p ) ^ p + \min \big \{  \frac{ (\caseven) ^ 2 }{(8 \cdot 256)^2}  
    \eta _\Om ^ 2 , 1  \big \} \,.$$ 
  Indeed, assume that $D_\Om \leq 1$ and $\eta _\Om \leq \frac 2 3 \eta _0$.  If $d_p$ is defined by \eqref{f:dp}, two cases may occur. If $D _\Om \leq d _p$, it holds
  $$\mu _p (\Om, \phi) \geq \frac{(\pi _p) ^ p }{d ^p} \geq (\pi_p) ^ p + 1 \,.$$  
 On the other hand, if $D_\Om \geq d_p$,  since $d _p \geq \frac{2}{3}$, we have
 $$\frac{\eta _\Om}{D_\Om} \leq \frac{3}{2} \eta _\Om \leq \eta _0\,.$$ 
  By Lemma \ref{l:lemma1},       we obtain
$$\mu _ p  \big (\frac{ \Om  }{ D_\Om},  \phi  (D_\Omega \; \cdot) \big )  \geq   (\pi _ p ) ^ p +   \frac{ (\caseven) ^ 2 }{(8 \cdot 256)^2}  
     \Big ( \frac{\eta _\Om}{D_\Om} \Big ) ^ 2 
\geq  (\pi _ p ) ^ p +  \frac{ (\caseven) ^ 2 }{(8 \cdot 256)^2}  
      \eta _\Om ^2\,. $$  

\medskip 

{\bf Step 2}. It holds
 $$D_\Om \leq 1 \, ,  D_\Om \cdot \eta _\Om \leq \frac 1  3 \eta _0 \ \Rightarrow \ \mu _p (\Om, \phi) 
 \geq (\pi _p ) ^ p + \min \big \{   \frac{ (\caseven) ^ 2 }{(8 \cdot 256)^2}  
    \eta _\Om ^ 2 , 1  \big \} \,.$$ 
  Indeed, assume that $D_\Om \leq 1$ and $ D_\Om \cdot \eta _\Om \leq \frac 1  3 \eta _0$. Two cases may occur. 
  If  $\eta _\Om \leq \frac 2 3 \eta _0$, we are done by Step 1. If  $\eta _\Om \geq \frac 2 3 \eta _0$,  we have $D _\Om \leq \frac{1}{2} \leq d _p$, and hence
  $$\mu _p (\Om, \phi)  \geq \frac{(\pi _p ) ^ p }{d_p} 
 \geq (\pi _p ) ^ p + 1 \,.$$  

\medskip
{\bf Step 3.} It holds
$$D _\Om = 1 \ \Rightarrow \mu_p (\Omega, \phi)\ge  ({\pi_p}) ^ p   +     \frac{ (\caseven) ^ 2  \eta _0 ^ 2 }{(7 \cdot 16 \cdot 256  )^2}  \eta  _\Om^2\,.$$

Indeed, let $n= [ \frac{6}{\eta _0} ] + 1$, so that $\frac{2}{n} \leq \frac{1}{3} \eta _0$. Let $\{\Om _1, \dots, \Om _n \} $ be a $\phi$-weighted 
measure equipartition  of  a first eigenfunction for $\mu  _p (\Om, \phi)$ in $\Om$. We have

$$D_{\Om _i }  \cdot \eta _{\Om _i}  \leq 2 |\Om _i |  =  \frac{2}{n}  |\Om |   \leq \frac{1}{3} \eta _0 D _\Om \eta _\Om \leq \frac{1}{3} \eta _0  \,.$$ 
Thus, by \eqref{f:meanvalue} and  Step 2, we have 

$$\mu _p (\Om, \phi) \geq \min _{i} \mu _p (\Om_i, \phi) \geq  (\pi _p ) ^ p + \min \big \{   \frac{ (\caseven) ^ 2 }{(8 \cdot 256)^2}  
     \eta _{\Om  _i} ^ 2 , 1  \big \} $$ 

Since
$$\eta _{\Om _i} \geq |\Om _i|   = \frac{1}{n} |\Om  | \geq \frac{1}{2n } \eta _\Om D _\Om  =  \frac{1}{2n } \eta _\Om \,.$$ 
 we infer that 
 $$\begin{aligned} \mu _p (\Om, \phi) & \geq  (\pi _p ) ^ p + \min \big \{   \frac{ (\caseven) ^ 2 }{(16 \cdot 256 )^2}  \frac{ \eta _\Om ^ 2 }{n ^ 2} 
   , 1  \big \} 
 \\ \noalign{\smallskip} & = 
   (\pi _p ) ^ p + \min \big \{    \frac{ (\caseven) ^ 2 }{(16 \cdot 256 )^2}  \frac{ \eta _\Om ^ 2 }{( [ \frac{6}{\eta _0} ] + 1  ) ^ 2 } 
   , 1  \big \} 
 \\ \noalign{\smallskip} & \geq 
   (\pi _p ) ^ p + \min \big \{  \frac{ (\caseven) ^ 2 }{(16 \cdot 256 )^2}  \frac{1}{( [ \frac{6}{\eta _0} ] + 1  ) ^ 2 } 
   , 1  \big \}    \eta _\Om ^ 2  
   \\ \noalign{\smallskip} &= 
  (\pi _p ) ^ p +   \frac{ (\caseven) ^ 2 }{(16 \cdot 256 )^2}  \frac{1}{( [ \frac{6}{\eta _0} ] + 1  ) ^ 2 }   \eta _\Om ^ 2   \,.
      \end{aligned} 
  $$  
  
  Finally we observe that 
  
  $$\big [ \frac{6}{\eta _0}  \big ] + 1 \leq \frac{6}{\eta_0} + 1 \leq \frac{6 + \eta _0} {\eta_0} \leq \frac{7}{\eta_0}\,,$$
  so that
  $$ 
 \mu _p (\Om, \phi) \geq 
 (\pi _p ) ^ p + 
    \frac{ (\caseven) ^ 2  \eta _0 ^ 2 }{(7 \cdot 16 \cdot 256  )^2} \eta _\Om ^2   
 $$ 
 
 \qed

\bigskip \bigskip 
{\bf Proof of Theorem \ref{t:Ngap} in dimension $N \geq 3$}. 
The passage to higher dimensions  follows exactly as in \cite[Section 6]{ABF24}, as a consequence of the two dimensional case
(and it is also likely that the involved constants might be tracked). 
\qed

 \section{Proof of Theorem \ref{t:stability} }\label{sec:proof2} 	
		 Let  $m \in \N$ be such that $\alpha \ge \frac 1m$. Assume by contradiction that, in the limit as $n \to + \infty$, the sequence $\{\Omega_n\}$ loses less than $(N-1)$ dimensions i.e.,  for some $l >1$, 
	 $$\lim _{n \to + \infty} a _ i ^n >0 \qquad \forall i \leq l\,.$$ 
	 Assume, without loosing generality, that ${\|\phi_n\|_\infty}=1$. 
	 Then,   there exist  an open convex set $\omega \subset \mathbb{R}^l$ and a $\big(\frac{1}{N-l +m}\big)$-concave function $\phi: \omega \to \mathbb{R}_+$, such that  	
	$$
	\pi _p ^p = \lim_{n \to + \infty}  \mu_p(\Omega_n, \phi _n) = \mu_p(\omega, \phi) > \pi_p ^p,
	$$    
	
	where the first equality holds by assumption, the second equality holds by Proposition \ref{abf209}
 (and Remark \ref{abf210}),
and the strict inequality holds by Proposition \ref{ambu20}. This proves statement (i), i.e.,  $a_i^n \to 0 \quad  \forall i \geq 2$.  From the proof of Proposition \ref{abf209}, we infer that	$\Om _n$  converge in Hausdorff distance to an  (open) interval $I$ of unit length. Moreover, subtracting sequences if necessary,  there exists a nonempty convex set $\Om\sq \R^N$  and a $\alpha$-concave function $\phi: \Om\to (0,1]$ with $\|\phi\|_\infty=1$ such that 
$$ T_n(\Om_n) \stackrel{H}{\longrightarrow}  \Om \qquad \mbox{ and } \qquad \phi_n (T_n^{-1}(x)) \to \phi(x)  \mbox{ for  a.e.  } x \in  \Om\,, $$  
and the sets  
$$
	\widetilde{\Omega}_n  = \big \{ (x,y) \in \R^{N+m} \ :\  \, x \in  T _n(\Omega_n) , \, \norma{y}_{\R^m} <\omega_m^{-\frac 1 m} \phi_n^{\frac{1}{m} }(T_n^{-1}(x))  \big \}
$$
  converge   to the  convex set $\widetilde \Om   \subset \R ^ {N+m}  $,   with   $I = \pi _ { \R \times \{0 \} ^ { N  -1 +m } }  (\widetilde \Om)$,  given by 
$$
	\widetilde{\Omega}  = \big \{ (x,y) \in \R^{N+m} \ :\  \, x \in  \Omega, \, \norma{y}_{\R^m} <\omega_m^{-\frac 1 m} \phi^{\frac{1}{m} }( x)  \big \}.
$$
  Moreover, setting  	
	 \begin{equation}\label{f:sections}\begin{aligned} 
	 S_{x_1}& := \big \{ x \in \widetilde \Om\ :\   \pi _ { \R \times\{0 \} ^ { N -1 +m }    }    (x) = x_1 \big \} \\
	 \noalign{\smallskip} &  =  \big \{(x_1,x',y):  (x_1,x') \in \Om, \norma{y}_{\R^m} <\omega_m^{-\frac 1 m} \phi^{\frac{1}{m} }( x_1,x')\big  \} \, , \end{aligned}\end{equation}

and 	 $$
	 \Phi (x_1):= \mathcal{H}^{  N-1+m   }  (S_{x_1})\,, $$
	we have	$$
	\pi _p ^p = \lim_n \mu_p(\Omega_n, \phi _n) = \mu_p(I,  \Phi ) \geq \pi_p^p\,.
	$$
We claim that the equality $\mu_p(I,  \Phi ) = \pi_p^p$ implies that $ \Phi $ is constant. For $p= 2$, this was already observed in  \cite[Section 2]{PW}.   For arbitrary $p \in (1, + \infty)$, this follows from inspection of the proof of Proposition \ref{p:1d}: precisely,  from the equality \eqref{f:inspection}, we see that, if $\mu_p(I,  \Phi ) = \pi_p^p$   the term $|\psi - \kappa|$ therein must vanish, namely that $(\log \Phi )' = \kappa$ for some $\kappa \leq 0$.
	In turn,  from inspection of the proof of Lemma 2.2 in \cite{FNT12}, we see that 
	the equality $\mu_p(I, \Phi ) = \pi_p^p$ implies that $\kappa = 0$. (More precisely one has to observe that, when $\mu_p(I,  \Phi ) = \pi_p^p$,  the inequality after fomula (2.4) in  \cite{FNT12}  holds with equality sign, which forces $\kappa$ to vanish.)

	Now, from the constancy of $\Phi $, we deduce that the following concavity inequality, due to Brunn-Minkowski Theorem,
$$		 \Phi ^{\frac{1}{ N-1+m  }}(tx_1+(1-t)x_2) \geq t \Phi ^{\frac{1}{  N-1+m  }}(x_1)+(1-t)\Phi ^{\frac{1}{  N-1+m   }}(x_2) \qquad \forall x_1, x_2 \in I\, , \forall t \in [0, 1]\,,
	$$
holds with equality sign, which is known to occur if and only if $S_{x_1}$ and $S_{x_2}$ agree up to a translation and a dilation  (see \cite{Schneider}). We conclude that all the sections of $\widetilde \Om$ are translations of a fixed one, say $S_0$, namely there exists a function $\tau :I \to \R ^{ {N+m}  }$, with $\tau (0) = 0$, such that 
$$S_ {x_1} = S _0 + \tau (x_1)\,.$$

By  the convexity of $\widetilde \Omega$, the function $\tau = \tau ( x_1)$ has to be linear, namely  there exists  a vector $v \in \R ^ {  N+m  }$  such that 
$$S_{x_1}=  S_0 +  x_1 v \,.$$  
  We notice that actually  $v = 0$. 
  Indeed, as a consequence of the above equality, combined with the structure of the sets $S_{x_1}$ (cf.\ \eqref{f:sections}),  all the $x_1$-sections of $\Om$, namely the sets of points $(x_1,x')$ belonging to $\Om$,  are translations of a fixed set by the vector $v$. This implies that $v = 0$, in view of the fact that the John ellipsoid of $\Om_n$ (and hence of $\Om$) is centred at the origin, with the axes along the cartesian directions.  Finally, the equality $S_{x_1} = S_0$ holding for every $x_1 \in I$, again combined with \eqref{f:sections}, yields that the function $\phi$ is independent of the $x_1$ variable. \qed

\bigskip
\bigskip
\bigskip 

\bibliographystyle{mybst}

\bibliography{References}

\end{document}